\documentclass[10pt]{amsart}

\numberwithin{equation}{section}

\usepackage{amscd}

\newtheorem{thm}{Theorem}[section]
\newtheorem*{tthm}{Theorem}

\newtheorem{lem}[thm]{Lemma}
\newtheorem{cor}[thm]{Corollary}
\newtheorem{prop}[thm]{Proposition}

\newtheorem*{claim}{Claim}
\newtheorem*{conjecture}{Conjecture}
\newtheorem*{claim1}{Claim 1}
\newtheorem*{claim2}{Claim 2}

\theoremstyle{definition}
\newtheorem{defin}[thm]{Definition}
\newtheorem{ex}[thm]{Example}
\newtheorem*{question}{Question}

\theoremstyle{remark}
\newtheorem{rmk}[thm]{Remark}

\DeclareMathOperator{\card}{card}
\DeclareMathOperator{\notdiv}{\not |}
\DeclareMathOperator{\ord}{ord}
\DeclareMathOperator{\Irr}{Irr}
\DeclareMathOperator{\codeg}{codeg}

\begin{document}
\title{A generalization of Dumas-Eisenstein criterion}
\author[Boris \v Sirola]{Boris \v Sirola}
\address{Department of Mathematics, University of Zagreb, Bijeni\v cka 30, 10000 Zagreb, Croatia}
\email{sirola@math.hr}


\begin{abstract}
We introduce an interesting and rather large class of monoid homomorphisms, on arbitrary integral domain $R$, that we call Dumas valuations. Then we formulate a conjecture addressing the question asking when a polynomial $f\in R[X]$ cannot be written as a product $f=gh$ for some nonconstant polynomials $g,h\in R[X]$. The statement of the conjecture presents a significant generalization of the classical Eisenstein-Dumas irreducibility criterion. In particular our approach can be very useful while studying the irreducibility problem for multivariate polynomials over any integral domain and polynomials over orders in algebraic number fields. We provide a strong evidence that our conjecture should be true.
\end{abstract}

\maketitle

\section*{Introduction}

Every ring we consider is commutative having an identity. Given a ring $R$, by $R^{\times}$ denote the set of its nonzero elements and by $R^{\ast}$ the group of its invertible elements. The set of positive integers is denoted by $\mathbb N$, while we define $\mathbb N_0=\mathbb N\cup \{ 0\}$.

A very interesting problem in commutative algebra is to understand the set of irreducible elements $\Irr R$ for a given integral domain $R$. In particular for various interesting integral domains $A$ we consider the rings of polynomials $A[X]=A[X_1,\ldots ,X_k]$, in $k\geq 1$ variables. Then for a certain $f\in A[X]$ one would often like to know whether it is irreducible or not. Let us emphasize that in what follows the meaning of the phrase ``$f$ is irreducible'' is that $f$ cannot be written as a product of two nonconstant polynomials of $A[X]$.
Perhaps the first result of such kind is due to Gauss who proved the irreducibility of cyclotomic polynomials. His result was years later generalized, first by Sch\" onemann \cite{Sc} and then by Eisenstein who established a well known, more special but simple and very applicable, criterion \cite{Ei}. This Eisenstein's criterion was, again years later, further generalized by Dumas \cite{Du} via the Newton polygons method; see also \cite{Ku,Or1,Or2,Or3,Re}, and \cite{Mo} for a nice survey of the method. Let us state that stronger result in its more or less original form. (Having in mind the well known Gauss lemma, the ring $\mathbb Z$ in the theorem can be replaced with any UFD.) For that purpose recall the $p$-order map $\ord _p:\mathbb Z\to \mathbb N_0\cup \{ +\infty\}$, defined as usual; where in particular $\ord _p(z)=+\infty$ if and only if $z=0$.

\begin{tthm} \textup{(Eisenstein-Dumas irreducibility criterion)}

\noindent
Let $f=c_nX^n+\cdots +c_1X+c_0$ be a polynomial in $\mathbb Z[X]$. Suppose there is a prime $p\in\mathbb N$ such that $\ord _p(c_n)=0$, $\gcd (n,\ord _p(c_0))=1$ and $(n-j)\ord _p(c_0) \leq n\ord _p(c_j)$ for $j=1,\ldots ,n-1$. Then $f$ is irreducible over $\mathbb Q$.
\end{tthm}

\noindent
The above criterion was further generalized, based on ideas due to Sch\" onemann, Eisenstein and Dumas, so that $\mathbb Q$ is replaced with any valued field; see for example \cite{Br,KhSa,MacL,PZ}.
In general the method using Newton polygons is a very powerful technique, both in number theory and algebraic geometry. For example using it one can further treat reducibility questions for various interesting polynomials; see for example \cite{Fi,FiTr}.
For a relatively recent work on irreducibility of multivariate polynomials with coefficients from any field, via Newton polytopes method inspired by the work of Ostrowski \cite{Os1,Os2}, see \cite{Ga}.

The crucial idea of our approach in this work is to introduce a certain interesting and large enough class of monoid homomorphisms that we call Dumas valuations, and more generally $\Gamma$-Dumas valuations. There are two kinds of such maps; see Sect. \ref{prelim} for details. For an integral domain $A$ let $A[X]$ be the corresponding ring of polynomials, as above. Then for Dumas valuations of the first kind on the ring $R=A[X]$ we will have various polynomial degree maps. A typical Dumas valuation of the second kind arises as follows. We take a prime element $p$ of $R$, when $R$ is a Noetherian integral domain or a UFD, and consider the $p$-order map $\ord _p:R\to \mathbb N_0\cup \{ +\infty\}$. But we also have another interesting family of Dumas valuations of the second kind given via the below introduced codegree maps.

The central theme of this work is to provide a strong evidence that the following conjecture holds. Note that its statement, for $\omega$ being a Dumas valuation of any kind, presents a significant generalization for a number of so far established results addressing the irreducibility questions for polynomials in finitely many variables, with coefficients from some integral domain.
\begin{conjecture}
Let $R$ be an integral domain. Let $n>1$ be a natural number and elements $\varphi _0,\varphi _1,\ldots ,\varphi _n\in R$. Let $\omega$ be a Dumas valuation on $R$ and assume the following:
\begin{itemize}
\item[({\sf A})] $\omega (\varphi _n)=0$, $\omega (\varphi _0)\in\mathbb N$ and $\gcd (n,\omega (\varphi _0))=1$.
\end{itemize}
Further assume one of the following two possibilities: either
\begin{itemize}
\item[({\sf B1})] $\omega$ is a Dumas valuation of the first kind and we have the inequalities
    $$
    n\omega (\varphi _j)\leq (n-j)\omega (\varphi _0) \quad \text{ for $j=1,\ldots ,n-1$};
    $$
\end{itemize}
or
\begin{itemize}
\item[({\sf B2})] $\omega$ is a Dumas valuation of the second kind and we have the inequalities
    $$
    n\omega (\varphi _j)\geq (n-j)\omega (\varphi _0)  \quad \text{ for $j=1,\ldots ,n-1$}.
    $$
\end{itemize}

Then the polynomial
$$
f=\varphi _nX^n+\cdots +\varphi _1X+ \varphi _0
$$
cannot be written as a product $f=gh$, for some  $g,h\in R[X]$ of degrees at least one; i.e., $f$ is irreducible.
\end{conjecture}

Observe that the above given conditions ({\sf B1}) and ({\sf B2}) come as no surprise. In fact they are straightforwardly motivated by the mentioned seminal work due to Dumas. In particular a very special case of the condition ({\sf B1}) appears in \cite{PS}, where the authors treat the irreducibility of bivariate polynomials; see also \cite[\S 1.8]{MiS}. But let us emphasize that for their results it is crucial that their polynomials are defined over an algebraically closed base field. The main result of this paper is the following theorem. Let it be said that the case of a Dumas valuation of the first kind is somewhat more difficult to treat, but it turns out that to be a worthwhile effort. At the same time, for a considerable part of our examples for $\omega$ being of the second kind the considered setting is in its essence similar to the classical one due to Eisenstein and Dumas.

\begin{thm}
\label{mainTHM}
The above Conjecture is true in the following cases:
\begin{itemize}
\item[(I)] $\omega$ is a Dumas valuation of the first kind and $n\leq 7$;
\item[(II)] $\omega$ is a Dumas valuation of the second kind and $n\leq 6$.
\end{itemize}
\end{thm}

Concerning our work observe the following. Given any integral domain $A$ we can define the quotient field $\mathcal Q(A)$. In particular if $A$ is a UFD, then any polynomial over $A$ which is reducible over $\mathcal Q(A)$ will be reducible over $A$ as well. But if $A$ is an arbitrary integral domain, then one in fact does not necessarily have the last conclusion. Here we have in mind for example the rings $A$ which are orders in the ring of integers for an algebraic number field $\mathbb K$, but are not UFDs. In other words $A$ is either a non-maximal order in $\mathbb K$, or it is the maximal order $\mathcal O_{\mathbb K}$ when the class number $h_{\mathbb K}$ is greater than $1$.  Thus our general approach cannot be, at least in some part, simply deduced from the existing methods where the considered polynomials have their coefficients from a field.

Before we provide more details it might be useful to look at two simple but rather motivating examples. Let $A[X]$ be the ring of polynomials in $k\geq 2$ variables $X_i$, with coefficients in an integral domain $A$. Suppose $\varphi _n,\varphi _0 \in A[X]$ are such that the constant term of $\varphi _n$ is nonzero while the constant term of $\varphi _0$ is zero. Let also the lowest homogeneous component of $\varphi _0$ be of order $\ell\in\mathbb N$; see Remark \ref{prelim-rmk2}. Then for either $n\in \{ 2,4\}$ and $\ell$ odd or $n\in \{ 1,3,5\}$ and $\ell$ even, the polynomial
$$
f=\varphi _n(X_1,\ldots ,X_k)Y^n+ \varphi _0(X_1,\ldots ,X_k)
$$
is irreducible in the polynomial ring $A[Y,X_1,\ldots ,X_k]$. On the other hand suppose there is some $j\in \{ 1,\ldots ,k\}$ so that moreover $\varphi _n\in A[X_1,\ldots ,\widehat{X_j},\ldots ,X_k]$; i.e., there is no $X_j$ in $\varphi _n$. Also let the degree of $\varphi _0$ with respect to the variable $X_j$ be equal to the natural number $\ell$. Then for either $n\in \{ 2,4\}$ and $\ell$ odd or $n\in \{ 1,3,5,7\}$ and $\ell$ even the above defined polynomial $f$ is irreducible in $A[Y,X_1,\ldots ,X_k]$.

Let us explain the main idea of our method when we want to show that particular multivariate polynomials are irreducible over some integral domain $A$. Suppose we have some $A$-polynomials $\varphi _0,\varphi _1,\ldots ,\varphi _n$ in $k\geq 1$ variables; i.e., belonging to the ring $R=A[X_1,\ldots ,X_k]$. Then consider the polynomial $f=\sum _{i=0}^n \varphi _iX^i\in R[X]$. We try to find a convenient $\boldsymbol{\nu}\in\mathbb N_0^k$ such that for the Dumas valuation of the first kind $\omega =\deg _{\boldsymbol{\nu}}$ on $R$ we have fulfilled both the conditions {\sf (A)} and {\sf (B1)} of our Conjecture; see Definition \ref{nu}. Then in particular for $n\leq 7$ we know that $f$ cannot be written as a product of two nonconstant polynomials in $A[X,X_1,\ldots ,X_k]$. In order to get the first glimpse over the large class of multivariate polynomials for which the irreducibility problem can be treated in such a way let us present here the following example of this kind; see Corollary \ref{Dumas1-cor} for a more general result, and also both Example \ref{Dumas1-ex} and Remark \ref{Dumas1-rmk}.

\begin{ex}
\label{introd-ex}
Let $A$ be any integral domain, $n>1$ an integer and $\varphi _1,\ldots ,\varphi _{n-1}$ and $\psi _0$ be polynomials of the ring $A[Y,Z]$ satisfying the following. For the $Z$-degree we have $\deg _Z\psi _0\leq \mu -1$, where $\mu \in\mathbb N$ is such that $\gcd (n,\mu )=1$. Also assume that
$\deg _Z\varphi _i \leq \bigl\lfloor \frac{(n-j)\mu}{n}\bigr\rfloor$ for every $1\leq j<n$. Then for $n\leq 7$ and any $a,b\in A^{\times}$ and $t\in\mathbb N_0$, the polynomial
$$
f=aX^n +\varphi _{n-1}(Y,Z)X^{n-1} +\cdots +\varphi _1(Y,Z)X +bY^tZ^{\mu} +\psi _0(Y,Z)
$$
is irreducible in the ring $A[X,Y,Z]$.
\end{ex}

There is a ``dual approach'' via particular Dumas valuations of the second kind. Namely suppose the ring of polynomials $R$ and $f\in R[X]$ are as above. Now we try to find a convenient  $\boldsymbol{\nu}\in\mathbb N_0^k$ such that for the Dumas valuation of the second kind $\omega =\codeg _{\boldsymbol{\nu}}$ on $R$ we have both the conditions {\sf (A)} and {\sf (B2)} of the Conjecture; see Definition \ref{codeg-nu}. The following example illustrates this approach; see Corollary \ref{Dumas2-cor1} and Example \ref{Dumas2-ex1}.

\begin{ex}
\label{introd-ex2}
Let $A$, $n$ and $A[Y,Z]$ be as in the previous example. Suppose we have polynomials $\varphi _1,\ldots ,\varphi _n$ and $\psi _0$ from $A[Y,Z]$ satisfying the following: {\sf (1)} The constant term of $\varphi _n$ is nonzero; {\sf (2)} The constant term of $\varphi _j$ equals zero, for each $1\leq j<n$; {\sf (3)} The total degree of each monomial of $\psi _0$ is greater than $1$. Then for $n\leq 6$ and any $a,b\in A$ that are not both equal to zero, the polynomial
$$
f=\varphi _n(Y,Z)X^n +\varphi _{n-1}(Y,Z)X^{n-1} +\cdots +\varphi _1(Y,Z)X +aY+bZ +\psi _0(Y,Z)
$$
is irreducible in the ring $A[X,Y,Z]$.
\end{ex}

\begin{rmk}
\label{rmk0}
Of course we can consider the above two examples in the case when the polynomials $\varphi _j$ (and $\psi _0$) are in one variable, say $Y$. Thus we have a polynomial $f(X,Y)=\sum _{j=0}^n \varphi _j(Y)X^j$ in which we are interested to know whether it is irreducible over $A$ or not. But observe that $Y$ is a prime element of the ring $R=A[Y]$ and therefore the conditions {\sf (1)--(3)} of the latter example boil down to the following: {\sf (1)} $\ord _Y\varphi _n=0$; {\sf (2)} $\ord _Y\varphi _j\geq 1$ for each $1\leq j<n$; {\sf (3)} $\ord _Y\varphi _0=1$. In other words, at least when $A$ is a UFD, here works the classical Eisenstein criterion; or, in a more general setting, the before formulated Eisenstein-Dumas criterion. This means that our method via codegree function in fact gives something new only for polynomials in three and more variables.
\end{rmk}

The second main application of our work is to the reducibility problem for polynomials in one variable with coefficients from an order $R$ in an algebraic number field. Here we need to have a large enough supply of prime elements in $R$; i.e., elements of the set $\Pr R$ of all primes in $R$. As a useful result related to that, which is interesting in its own right, we have Lemma \ref{Orders-lem1} proved in \cite{Sch}. In Section \ref{Orders} of the paper we give several interesting results which explain how in various settings one can show that a particular polynomial $f\in R[X]$, of degree less than $7$, is irreducible over $R$. The following theorem collects the three main results of that last section; see Theorems \ref{Orders-thm1} and \ref{Orders-thm2} and Proposition \ref{Orders-prop2}.

\begin{thm}
\label{sect4-mainTHM}
Let $\mathbb K$ be an algebraic number field and $R$ be an order in $\mathbb K$. Also let $N=N_{\mathbb K|\mathbb Q}: \mathbb K\to \mathbb Q$ be the corresponding norm map. Suppose 
$$
f=c_nX^n+ \cdots +c_1X+c_0
$$
is a polynomial in $R[X]$ of degree $n\leq 6$. 
\begin{enumerate}
	\renewcommand{\labelenumi}{(\roman{enumi})}
 \item Suppose there is some $\mathfrak p\in \Pr R$ such that $\ord _{\mathfrak p}(c_n) 
       =0$,  $\gcd (n,\ord _{\mathfrak p}(c_0))=1$ and $(n-j)\ord _{\mathfrak p}(c_0) \leq n\ord _{\mathfrak p}(c_j)$ for $j=1,\ldots ,n-1$. Then $f$ is irreducible over $R$.
 \item Let $\mathbb K$ be moreover a normal field of degree $m$ with its Galois group 
       $G_{\mathbb K}$ and suppose $R$ is $G_{\mathbb K}$-invariant. (For example this holds for any order $R$ in any quadratic field $\mathbb K$, or for $R=\mathcal O_{\mathbb K}$ when $\mathbb K$ is any normal field.) Let $p\in \Pr \mathbb Z$ be  such that there is some $\mathfrak p\in R$ satisfying $N(\mathfrak p)=p$; and so $\mathfrak p\in \Pr R$. Further suppose that $\gcd (n,m)=1$ and for $\ell =\ord _p(c_0)$ we have the following: {\sf (1)} $\ord _p(N(c_n))= 0= \ord _p(N(c_0/p^{\ell}))$; {\sf (2)} $\gcd (n,\ell )=1$; {\sf (3)} $(n-j)\ell \leq n\ord _p(c_j)$ for $j=1,\ldots , n-1$. Then $f$ is irreducible over $R$.
 \item Let $\delta \in\mathbb Z$ be cube-free and $\alpha$ be the real cube root of
       $\delta$. Consider the order $R=\mathbb Z[\alpha ]$ in the pure cubic field $\mathbb K=\mathbb Q(\alpha )$. Assume that $3\neq p\in \Pr \mathbb Z$ is such that $\gcd (p,\delta )=1$ and there is some $\mathfrak p\in \Pr R$ satisfying $N(\mathfrak p)=p$. Also suppose that if $\mathfrak p=a+b\alpha +c\alpha ^2$, for some $a,b,c\in \mathbb Z$, then we have that $p\notdiv ab-c^2 \delta$ or $p\notdiv ac-b^2$. Further let $\ell$ be as in \textup{(ii)} and assume the conditions {\sf (1)}\! -- {\sf (3)} as there hold here as well. Then $f$ is irreducible over $R$.
\end{enumerate}
\end{thm}
\noindent 
Observe in particular that (i) of the theorem gives a significant generalization of the before stated Eisenstein-Dumas criterion; of course, when $n\leq 6$. Besides we believe that one can obtain some other useful specific results like the statement (iii) of the theorem is.

Throughout the paper the following notation will be fixed. Let $f\in R[X]$ be as in the statement of the above Conjecture. Our task is to show that under the given assumptions there are no nonconstant polynomials  $g,h\in R[X]$ such that $f=gh$. We will do that via some auxiliary results that will sometimes be formulated in a more general form than we actually need for a proof of our theorem. (But we hope that this might give an idea how to generalize our approach and then perhaps obtain a complete proof of the above formulated Conjecture.) For that purpose in the sequel we will always assume that
$$
1\leq \deg h=t\leq s=\deg g,
$$
and write
\begin{equation}
\label{gh-poly}
g=a_sX^s+\cdots +a_1X+a_0 \quad \text{and} \quad h=b_tX^t+\cdots +b_1X+b_0,
\end{equation}
for some $a_i,b_j\in R$. As the starting point of our arguments we have the system of equalities for the polynomial coefficients arising from the assumed polynomial equality $f=gh$. That is we consider the equalities
$$
\varphi _0 =a_0b_0, \quad \varphi _1 = a_0b_1+a_1b_0, \quad \ldots \quad , \quad \varphi _{n-1} = a_{s-1}b_t + a_sb_{t-1}.
$$
Our main job is to quite precisely keep track about the values of a Dumas valuation of these polynomial coefficients. As it will be seen,  this is a rather complicated task.

\begin{rmk}
\label{rmk1}
As it is customary we always assume that the coefficients $a_i$ for $i>s$ or $i<0$, and $b_j$ for $j>t$ or $j<0$, are equal to zero.
\end{rmk}

The paper is organized as follows. In Sect$.$ \ref{prelim} we introduce the notions of $\Gamma$-Dumas valuations of the first kind and of the second kind. We also show that the class of these monoid homomorphisms is rather large. In Sect$.$ \ref{Dumas1} we prove Theorem \ref{THMDumas1kind} which in fact gives the part (I) of Theorem \ref{mainTHM}. We do this via three lemmas: Lemma \ref{degh=1-lemma}, Lemma \ref{degh=2-lemma} and Lemma \ref{degh=3-lemma}. Although their proofs are with elementary methods, they are quite involved and tedious. In particular the second one of these lemmas presents a more general result than we need for the proof of the mentioned part (I). In Sect$.$ \ref{Dumas2}  we prove Theorem \ref{THMDumas2kind} which settles the part (II) of Theorem \ref{mainTHM}. Similarly as for Theorem \ref{THMDumas1kind} we give its proof via three lemmas; the before proved Lemma \ref{degh=1-lemma} and two new ones, Lemma \ref{2Ddegh=2-lemma} and Lemma \ref{2Ddegh=3-lemma}. Again the second lemma states a more general result than we actually need. Besides let it be said that we can prove a stronger version of Lemma \ref{2Ddegh=3-lemma}. Namely as a fact we have that for a polynomial $f\in R[X]$ of degree $n=7$ there are no $g,h\in R[X]$ such that $f=gh$ and $\deg h=3$. In other words the part (II) of Theorem \ref{mainTHM} is valid for $n=7$ as well. But the argument we have for that is quite long and in a sense similar to the one for Lemma \ref{degh=3-lemma}. So we preferred not to include it in the paper. As it was mentioned before, in Sect$.$ \ref{Orders} we show how to use the part (II) of Theorem \ref{mainTHM} when we consider certain kind of polynomials in one variable having coefficients from some order of an algebraic number field. There we prove the above Theorem \ref{sect4-mainTHM}, via the mentioned two theorems and one proposition. Besides we illustrate our results by some interesting examples. Let us emphasize that all our arguments are purely algebraic.

At the end of this introduction a few more remarks are in order. Note that our work presented here is in a sense in accordance with the general philosophy that a significant portion of randomly chosen polynomials (in particular of high degrees) should be of irreducible ones. For more precise statements and some interesting recent results see for example \cite{B-SoK1,BV,OP}, where polynomials in one variable over $\mathbb Z$ were considered. For a related facts concerning bivariate $\mathbb Z$-polynomials see \cite{B-SoK2}.

As we already said it is interesting to understand the set of irreducible elements $\Irr R$ of a given integral domain $R$. And as a very important case one can take $R$ to be the ring of polynomials in finitely many variables over some integral domain $A$. Here it would be even more interesting to understand the subset of primes $\Pr R\subseteq \Irr R$. (Observe that $\Irr R=\Pr R$ if $A$ is a UFD.) Some other basic results toward this general problem were obtained in \cite{Sch}.

\section{Preliminaries on Dumas valuations}
\label{prelim}

Suppose $(\Gamma ,+,\succeq )$ is a totally ordered monoid which is a sub-monoid of the additive group $(\Pi, +)$ of some ring $(\Pi ,+,\cdot )$, with  zero $0=0_{\Pi}$. Moreover suppose that $x\succ 0$ for each nonzero $x\in\Gamma$ and that the total order $\succeq$ is monomial; i.e., if $u,v,w\in\Gamma$ are such that $u\succeq v$, then $u+w\succeq v+w$. Here the symbol $\succ$, of the strict order, has the standard meaning. Next, given $\Gamma$ as above, define the set $\Gamma ^{-\infty}=\Gamma \cup \{ -\infty\}$ and consider it as a totally ordered additive monoid, where for the symbol $-\infty$ we have both $k\succ -\infty $ for every $k\in \Gamma$ and $(-\infty )+k=-\infty$ for every $k\in \Gamma ^{ -\infty}$.  Analogously we consider a totally ordered additive monoid $\Gamma ^{+\infty}=\Gamma \cup \{ +\infty\}$, where now for the symbol $+\infty$ we have both $+\infty \succ k$ for every $k\in \Gamma$ and $(+\infty )+k=+\infty$ for every $k\in \Gamma ^{+\infty}$.

\begin{ex}
\label{prelim-Ex1}
Let $d\in\mathbb N$ and define the standard localization $\Pi =\mathbb Z_{(d)}$ at the element $d$. Consider the sub-monoid $\Gamma =\frac{1}{d}\mathbb N_0=\{ k/d \mid k\in\mathbb N_0\}$ of $(\Pi ,+)$, equipped by the standard total order $\geq$. The extended totally ordered monoids $\Gamma ^{\pm \infty}$ will be of special interest for us, with the ``basic cases'' $\Gamma ^{\pm \infty}=\mathbb N_0^{\pm \infty}$.
\end{ex}

The notion of a valuation on a field is standard in algebra; see \cite{FT}. The following is in fact its generalization.

\begin{defin}
\label{Dumas}
Assume $\Gamma =(\Gamma ,+,\succeq )$ is a totally ordered monoid where the order $\succeq$ is monomial. Let $R$ be an integral domain. A map
$$
\omega :R\to \Gamma ^{ -\infty}
$$
will be called a {\em $\Gamma$-Dumas valuation of the first kind} if for all $x,y\in R$ the following hold:
\begin{itemize}
\item[(D0)\;\; ] $\omega (x)\succeq 0_{\Gamma}$ if $x\neq 0$;
\item[(D1)\;\; ] $\omega (xy)=\omega (x) +\omega (y)$;
\item[(D2-1)] $\omega (x)=-\infty$ if and only if $x=0$;
\item[(D3-1)] $\omega (x+y)\preceq \max \{ \omega (x),\omega (y)\}$, where we have an equality if $\omega (x) \neq \omega (y)$.
\end{itemize}

A map
$$
\omega :R\to \Gamma ^{ +\infty}
$$
will be called a {\em $\Gamma$-Dumas valuation of the second kind} if for all $x,y\in R$ the above conditions (D0), (D1) and also the following two conditions hold:
\begin{itemize}
\item[(D2-2)] $\omega (x)=+\infty$ if and only if $x=0$;
\item[(D3-2)] $\omega (x+y)\succeq \min \{ \omega (x),\omega (y)\}$, where we have an equality if $\omega (x) \neq \omega (y)$.
\end{itemize}

Unless it is necessary to be precise, any of the two defined maps will be called in short a {\em $\Gamma$-Dumas valuation}. In particular when $\Gamma =\mathbb N_0$ we will say that the corresponding $\omega$ is a {\em Dumas valuation}.
\end{defin}

\begin{rmk}
\label{prelim-rmk1}
(1) Let $R$ be any integral domain. If we put $\omega (0)=-\infty$ and $\omega (x)=0_{\Gamma}$ for every $x\in R^{\times}$, then $\omega$ is a $\Gamma$-Dumas valuation of the first kind. Similarly by $\omega (0)=+\infty$ we obtain a $\Gamma$-Dumas valuation of the second kind. Such $\Gamma$-Dumas valuations, that are not of particular interest, will be called {\em trivial} ones. Therefore we will always assume that our $\Gamma$-Dumas valuation is {\em nontrivial}; i.e., that there is at least one $x_0\in R^{\times}$ such that $\omega (x_0)\neq 0_{\Gamma}$.

(2) Suppose $\Pi$ is a ring of characteristic zero. Let also $\Gamma =(\Gamma ,+)$ be a sub-monoid of $(\Pi ,+)$ containing the identity $1_{\Pi}$, and assume there is a monomial total ordering $\succeq$ on $\Gamma$. Again suppose that $x\succ 0$ for each nonzero $x\in\Gamma$. Next observe that $\imath :\mathbb Z\to \Pi$, $\imath (z)=z1_{\Pi}$, is a monomorphism of rings with identity; and thus we can in particular assume that the additive monoid $(\mathbb N_0,+)$ is a sub-monoid of $\Gamma$. Having such setting one could easily formulate a generalization of our Conjecture so that $\omega$ is now a $\Gamma$-Dumas valuation. Further details are left to the interested reader.
\end{rmk}

It might be useful to note here the following simple fact.

\begin{lem}
\label{ba0}
Let $\omega$ be a Dumas valuation on an integral domain $R$ and let $d\in\mathbb N$ be arbitrary. Define $\Gamma =\frac{1}{d}\mathbb N_0$ as in \textup{Example \ref{prelim-Ex1}}. Then the map $\omega _1=\frac{1}{d}\omega$, defined by $\omega _1(x)=\omega (x)/d$, is a $\Gamma$-Dumas valuation on $R$ of the same kind as $\omega$ is.
\end{lem}

An easy proof of the following basic lemma is also omitted. Let us emphasize that in particular its last three claims will be used many times in the sequel, often without explicit referring to them. Besides observe that if $\Gamma$ is included in some ring $\Pi$ as it was above, then the lemma is valid for $\Gamma$-Dumas valuations as well.

\begin{lem}
\label{ba1}
Let $\omega$ be a Dumas valuation defined on an integral domain $R$.
\begin{enumerate}
 \renewcommand{\labelenumi}{(\roman{enumi})}
 \item By considering $R$ as a multiplicative monoid, the map $\omega$ is a homomorphism of monoids satisfying $\omega (1)=0$.
 \item For every $x\in R$ we have $\omega (x)=\omega (-x)$.
 \item For $x,y\in R^{\times}$ we have $\omega (xy)\geq \max \{ \omega (x),\omega (y)\}$.
 \item Let $x_1,\ldots ,x_t,y\in R$ be some elements. If $\omega$ is a Dumas valuation of the first kind, then
         $$
         \omega (x_1+\cdots +x_t) \leq \max \{ \omega (x_1),\ldots ,\omega (x_t)\} .
         $$
         Also if we have $\omega (x_i)<\omega (y)$ for every $1\leq i\leq t$, then
     \begin{equation}
     \label{ba1-label1}
     \omega (x_1+\cdots +x_t+y)=\omega (y).
     \end{equation}
     Further if $\omega$ is a Dumas valuation of the second kind, then
       $$
       \omega (x_1+\cdots +x_t) \geq \min \{ \omega (x_1),\ldots ,\omega (x_t)\} .
       $$
       Also if we have $\omega (x_i)>\omega (y)$ for every $1\leq i\leq t$, then again we have \textup{(\ref{ba1-label1})}.
\end{enumerate}
\end{lem}

Our next task is to show how to obtain some interesting classes of Dumas valuations on certain classes of integral domains. First we need some preparation on rings of polynomials. Given a natural number $n\geq 2$ we consider the set of all multi-indices $\mathbb N_0^n$ and then fix a particular $\boldsymbol{\nu}=(\nu _1,\ldots ,\nu _n)\in \mathbb N_0^n$. Define the {\em $\boldsymbol{\nu}$-weight} of any $I=(i_1,\ldots ,i_n)\in \mathbb N_0^n$ as an integer
$$
|I|_{\boldsymbol{\nu}}=\nu _1i_1+\cdots +\nu _ni_n\in \mathbb N_0.
$$
For multi-indices $I$ as above and $J=(j_1,\ldots ,j_n)$ we put $I\prec _{\boldsymbol{\nu}}J$ if either $|I|_{\boldsymbol{\nu}}<|J|_{\boldsymbol{\nu}}$ or $|I|_{\boldsymbol{\nu}}=|J|_{\boldsymbol{\nu}}$ and there is some $t\in\{ 1,\ldots ,n\}$ such that $$
i_1=j_1, \ldots  , i_{t-1}=j_{t-1} \quad \text{and} \quad i_t<j_t.
$$
If $I=J$ or $I\prec _{\boldsymbol{\nu}} J$, we write $I\preceq _{\boldsymbol{\nu}}J$. Note that $(\mathbb N_0^n,\preceq _{\boldsymbol{\nu}})$ is a total ordering. And it is monomial; i.e., for $I,J,K\in \mathbb N_0^n$ such that $I\preceq _{\boldsymbol{\nu}}J$ we have $I+K\preceq _{\boldsymbol{\nu}}J+K$ as well.

Now let $A$ be a ring and $A[X]=A[X_1,\ldots,X_n]$ the corresponding ring of polynomials in $n$ variables $X_i$. Given a multi-index $I=(i_1,\ldots ,i_n)$ we write $X^I$ instead of $X_1^{i_1}\cdots X_n^{i_n}$, and so every $f\in A[X]$ can be written as $f=\sum _Ia_IX^I$ for some coefficients $a_I\in A$. The following definition generalizes the standard notion of the total degree for polynomials.

\begin{defin}
\label{nu}
For $n\geq 2$ and fixed $\boldsymbol{\nu}\in\mathbb N_0^n$ define the {\em $\boldsymbol{\nu}$-degree} of a nonzero polynomial $f=\sum _Ia_IX^I\in A[X]$ as a number
$$
\deg _{\boldsymbol{\nu}}f=\max \{ |I|_{\boldsymbol{\nu}} \mid a_I\neq 0\} .
$$
For the zero polynomial we define $\deg _{\boldsymbol{\nu}}0=-\infty$.
For $\boldsymbol{\nu}=(1,\ldots ,1)$ we obtain the standard total degree of a polynomial $f$ that will be denoted by $\deg f$. Also for $\boldsymbol{\nu}=(0,\ldots 0,1,0,\ldots ,0)$, where $1$ is on the $i$-th place, we use the notation $\deg _i$; i.e., this is the degree with respect to the variable $X_i$. In particular when we have three variables $X,Y$ and $Z$, we will write $\deg _X,\deg _Y$ and $\deg _Z$, respectively.
\end{defin}

\begin{rmk}
\label{prelim-rmk2}
Given a polynomial $f\in A[X]$ and $\boldsymbol{\nu}\in\mathbb N_0^n$, as in the previous Definition, for $k\in \mathbb N_0$ we can define $f^{(\boldsymbol{\nu},k)}$, its {\em $\boldsymbol{\nu}$-homogeneous component} of order $k$, as the sum of all monomials $a_IX^I$ such that $|I|_{\boldsymbol{\nu}}=k$.
\end{rmk}

The corollary given below is an immediate consequence of the following lemma which is a straightforward generalization of the well known one.

\begin{lem}
\label{lemV1}
Let $A$ be a ring and $f,g\in A[X]$ polynomials.
\begin{enumerate}
 \renewcommand{\labelenumi}{(\roman{enumi})}
 \item We have
       $$
       \deg _{\boldsymbol{\nu}}(f+g) \leq \max \{ \deg _{\boldsymbol{\nu}}f, \deg _{\boldsymbol{\nu}}g\} .
       $$
       In particular we have an equality if $\deg _{\boldsymbol{\nu}}f\neq \deg _{\boldsymbol{\nu}}g$.
 \item We have
       $$
       \deg _{\boldsymbol{\nu}}(fg) \leq \deg _{\boldsymbol{\nu}}f + \deg _{\boldsymbol{\nu}}g.
       $$
       In particular we have an equality if $A$ is an integral domain.
\end{enumerate}
\end{lem}

\begin{cor}
\label{corV1}
Let $A$ be an integral domain and $R=A[X]$ the corresponding ring of polynomials in $n$ variables. For any $\boldsymbol{\nu} \in\mathbb N_0^n$ the map $\deg _{\boldsymbol{\nu}}: R\to \mathbb N_0^{-\infty}$ is a Dumas valuation of the first kind.
\end{cor}

Next we explain how to get some Dumas valuations of the second kind. For that purpose first recall that given a ring  $R$, an element $p\in R^{\times}$ is {\em prime} if $p\not\in R^{\ast}$ and it satisfies the following condition: If $p|xy$ for some $x,y\in R$, then $p|x$ or $p|y$. By $\Pr R$ we denote the set of all primes in $R$.

\begin{lem}
\label{lemU1}
Let $R$ be a Noetherian integral domain and $p\in \Pr R$; more generally, $p$ is any non-invertible element of $R$. Then for every $x\in R^{\times}$ there is a unique $k\in\mathbb N_0$ such that $p^k|x$ and $p^{k+1} \notdiv x$; and then we write $\ord _p(x)=k$.
\end{lem}

\begin{proof}
Suppose we have some $x\in R^{\times}$ such that $p^k|x$ for every $k\in\mathbb N$. Then $x=p^ky_k$ for some elements $y_k\in R^{\times}$. As we have
\begin{equation}
\label{lemU1-lab1}
p^ky_k=x=p^{k+1}y_{k+1},
\end{equation}
it follows that $y_k=py_{k+1}\in (y_{k+1})$ and so $(y_k)\subseteq (y_{k+1})$ for every $k\in\mathbb N$. If we would have the equality $(y_k)=(y_{k+1})$, for some $k$, then $y_{k+1}=ry_k$ for some $r\in R$. But then (\ref{lemU1-lab1}) easily implies that $1=pr$, which is impossible. Thus we obtain that all the inclusions between the principal ideals $(y_k)$ are strict, which is again impossible.
\end{proof}

Concerning the following corollary recall that an integral domain can be Noetherian but not a UFD, and also it can be a UFD but not a Noetherian ring.

\begin{cor}
\label{corU2}
Let $R$ be an integral domain which is Noetherian or a UFD. For any $p\in \Pr R$ we can define a map $\ord _p:R\to \mathbb N_0^{+\infty}$, where in particular we put $\ord _p(0)=+\infty$. Then this map is a Dumas valuation of the second kind.
\end{cor}

\begin{proof}
If $R$ is a Noetherian integral domain, then by the previous lemma the map $\ord p$ is well defined. And for $R$ a UFD this is well known.

The conditions (D0), (D1) and (D2-2) of Definition \ref{Dumas} obviously hold. For the condition (D3-2) take any $x,y\in R^{\times}$ and write them as $x=p^ku$ and $y=p^{\ell}v$, where $p\notdiv u,v$ and $k\geq \ell$. Then the inequality in (D3-2) is satisfied. In particular if $k>\ell$, then
$$
\ord _p(x+y)= \ell = \min \{ \ord _p(x), \ord _p(y)\} ;
$$
as it has to be. Also if $x=0$ or $y=0$, then the validity of (D3-2) is clear.
\end{proof}

There is another interesting collection of Dumas valuations of the second kind, which is in a sense dual to the above considered collection of Dumas valuations of the first kind. As we already noted in Remark \ref{rmk0} these Dumas valuations are primarily suitable while treating the irreducibility for polynomials in more than two variables. Keeping the before introduced notation we have the following definition.

\begin{defin}
\label{codeg-nu}
For $n\geq 2$ and fixed $\boldsymbol{\nu}\in\mathbb N_0^n$ define the {\em $\boldsymbol{\nu}$-codegree} of a nonzero polynomial $f=\sum _Ia_IX^I\in A[X]$ as a number
$$
\codeg _{\boldsymbol{\nu}}f=\min \{ |I|_{\boldsymbol{\nu}} \mid a_I\neq 0\} .
$$
For the zero polynomial we define $\codeg _{\boldsymbol{\nu}}0=+\infty$. For $\boldsymbol{\nu}=(1,\ldots ,1)$ we will use the notation $\codeg$, and call it the {\em total codegree}. Analogously as for $\deg _i$ we define the $\codeg _i$.
\end{defin}

For the codegree function we have analogues of Lemma \ref{lemV1} and Corollary \ref{corV1}.

\begin{lem}
\label{lemZ1}
Let $A$ be a ring and $f,g\in A[X]$ polynomials.
\begin{enumerate}
 \renewcommand{\labelenumi}{(\roman{enumi})}
 \item We have
       $$
       \codeg _{\boldsymbol{\nu}}(f+g) \geq \min \{ \codeg _{\boldsymbol{\nu}}f, \codeg _{\boldsymbol{\nu}}g\} .
       $$
       In particular we have an equality if $\codeg _{\boldsymbol{\nu}}f\neq \codeg _{\boldsymbol{\nu}}g$.
 \item We have
       $$
       \codeg _{\boldsymbol{\nu}}(fg) \geq \codeg _{\boldsymbol{\nu}}f + \codeg _{\boldsymbol{\nu}}g.
       $$
       In particular we have an equality if $A$ is an integral domain.
\end{enumerate}
\end{lem}

\begin{proof}
We may assume that both $f$ and $g$ are nonzero polynomials. Let $\codeg _{\boldsymbol{\nu}}f =|I_0|_{\boldsymbol{\nu}}$; i.e., $f=a_{I_0}X^{I_0} +\sum _Ka_KX^K$, where $a_{I_0}\neq 0$ and $I_0\preceq _{\boldsymbol{\nu}}K$ for each $K$ such that $a_K\neq 0$. And let also $\codeg _{\boldsymbol{\nu}}g =|J_0|_{\boldsymbol{\nu}}$; i.e., $g=b_{J_0}X^{J_0} +\sum _Lb_LX^L$, where $b_{J_0}\neq 0$ and $J_0\preceq _{\boldsymbol{\nu}}L$ for each $L$ such that $b_L\neq 0$. Then $fg= a_{I_0}b_{J_0}X^{I_0+J_0} +h$, where $h\in A[X]$ has a form $h(X)= \sum _Mc_MX^M$ such that $I_0+J_0 \preceq _{\boldsymbol{\nu}} M$ for each $M$ such that $c_M\neq 0$. In particular if $a_{I_0}b_{J_0}\neq 0$, then
$$
\codeg _{\boldsymbol{\nu}} (fg) =|I_0+J_0|_{\boldsymbol{\nu}}= |I_0|_{\boldsymbol{\nu}} + |J_0|_{\boldsymbol{\nu}} = \codeg _{\boldsymbol{\nu}}f + \codeg _{\boldsymbol{\nu}}g.
$$
Thus the part (ii) of the lemma is clear.

Suppose for example that $\codeg _{\boldsymbol{\nu}}f <\codeg _{\boldsymbol{\nu}} g$. Then obviously
$$
\codeg _{\boldsymbol{\nu}} (f+g) =|I_0|_{\boldsymbol{\nu}} = \codeg _{\boldsymbol{\nu}} f = \min \{ \codeg _{\boldsymbol{\nu}}f, \codeg _{\boldsymbol{\nu}}g\} ;
$$
and therefore (i) is clear as well.
\end{proof}

This corollary is a direct consequence of the above observed facts.

\begin{cor}
\label{corZ1}
Let $A$ be an integral domain and $R=A[X]$ the corresponding ring of polynomials in $n$ variables. For any $\boldsymbol{\nu} \in\mathbb N_0^n$ the map $\codeg _{\boldsymbol{\nu}}: R\to \mathbb N_0^{+\infty}$ is a Dumas valuation of the second kind.
\end{cor}

At the end of this section let us say that we believe there are more $\Gamma$-Dumas valuations that could be useful for various purposes.

\section{The case of Dumas valuations of the first kind}
\label{Dumas1}
The main purpose of this section is to prove the more interesting part of Theorem \ref{mainTHM}; i.e., the case when $\omega$ is a Dumas valuation of the first kind. In order to prove this we suppose to the contrary. That is we assume that for a polynomial as in our Conjecture there are some nonconstant polynomials $g$ and $h$ such that $f=gh$. Then we will obtain a contradiction. We begin with an obvious observation that will be used all the time.

\begin{lem}
\label{obvious lemma}
Let $\omega$ be a Dumas valuation on an integral domain $R$. Let $f\in R[X]$ be a polynomial as in the Conjecture, of degree $n>1$, and assume there are some nonconstant polynomials $g,h\in R[X]$ as in \textup{(\ref{gh-poly})} satisfying $f=gh$. Then
$$
\omega (a_s)=0=\omega (b_t).
$$
\end{lem}

Next we prove the following auxiliary result.

\begin{lem}
\label{degh=1-lemma}
Let $\omega$ be a Dumas valuation on an integral domain $R$. Let $f\in R[X]$ be a polynomial as in the Conjecture, of degree $n>1$. Then there are no polynomials $g,h\in R[X]$ such that $f=gh$ and $\deg h=1$.
\end{lem}

\begin{proof}
Assume that $\omega$ is of the second kind. Also suppose to the contrary, that there are some $g,h\in R[X]$ such that $f=gh$ and $\deg h=1$. Write
$$
g=a_{n-1}X^{n-1}+\cdots +a_1X+a_0 \quad \text{ and } \quad h=b_1X+b_0,
$$
for some coefficients $a_i,b_j\in R$. As a special case of the observation given in the previous lemma now we have that
\begin{equation}
\label{degh=1-lab1}
\omega (a_{n-1})=0=\omega (b_1).
\end{equation}
Here we also have the equalities
\begin{equation}
\label{degh=1-lab1A}
a_{k-1}b_1+a_kb_0=\varphi _k, \quad \text{ for $k=\{0,1,\ldots ,n-1\}$};
\end{equation}
see Remark \ref{rmk1}. Let us then show the following claim.

\begin{claim1}
The case $\omega (b_0)>\omega (\varphi _0)/n$ is impossible.
\end{claim1}
In order to see this suppose to the contrary. Then by (\ref{degh=1-lab1A}) for $k=n-1$, and using (\ref{degh=1-lab1}), we have that
\begin{equation}
\label{degh=1-lab2}
\omega (a_{n-2})=\omega (\varphi _{n-1}-a_{n-1}b_0) \geq \min \{ \omega (\varphi _{n-1}), \omega (b_0)\} \geq \omega (\varphi _0)/n.
\end{equation}
We proceed inductively. So assume we have proved that
\begin{equation}
\label{degh=1-lab3}
\omega (a_{n-j})\geq (j-1)\omega (\varphi _0)/n,
\end{equation}
for some $j$. Hence by (\ref{degh=1-lab1A}) for $k=n-j$ it follows that
$$
\omega (a_{n-j}b_0)= \omega (a_{n-j}) +\omega (b_0) > (j-1)\omega (\varphi _0)/n +\omega (\varphi _0)/n =j\omega (\varphi _0)/n,
$$
and further
$$
\omega (a_{n-j-1}) = \omega (\varphi _{n-j}-a_{n-j}b_0) \geq \min \{ \omega (\varphi _{n-j}), \omega (a_{n-j}b_0)\} \geq j\omega (\varphi _0)/n.
$$
So we have completed the inductive step. In particular for $j=n$ we obtain that
$$
\omega (a_0) \geq (n-1)\omega (\varphi _0)/n.
$$
And finally we deduce that
$$
\omega (\varphi _0)= \omega (a_0b_0)= \omega (a_0) +\omega (b_0) >\omega (\varphi _0);
$$
a contradiction. Thus we have proved Claim 1.

\vspace{1ex}

\begin{claim2}
The case $\omega (b_0)<\omega (\varphi _0)/n$ is impossible.
\end{claim2}
Again suppose to the contrary. Next we will derive a particular useful equality. For that purpose for every $j\in \{ 0,1,\ldots ,n-1\}$ let us multiply the equality (\ref{degh=1-lab1A}) by the factor $(-1)^j b_1^{n-j-1}b_0^j$. If we add up all the such obtained new equalities, it follows that
\begin{equation}
\label{degh=1-lab4}
\sum _{j=1}^{n-1} (-1)^{j+1}\varphi _j b_1^{n-j-1}b_0^j +(-1)^{n-1}a_{n-1}b_0^n = \varphi _0b_1^{n-1}.
\end{equation}
Having in mind Lemma \ref{ba1}(ii) and (\ref{degh=1-lab1}) observe that for the typical summand $\sigma _j= (-1)^{j+1}\varphi _j b_1^{n-j-1}b_0^j$, of the above sum, we have
\begin{equation}
\label{degh=1-lab5}
\omega (\sigma _j) =\omega (\varphi _j) +j\omega (b_0).
\end{equation}
Further by the inequality in the statement of Claim 2 and the assumption in ({\sf B2}) of the Conjecture we easily get the inequalities
$
(n-j)\omega (b_0) <\omega (\varphi _j).
$
Hence we deduce that
$$
\omega (\sigma _j) > n\omega (b_0)= \omega (a_{n-1}b_0^n), \quad \text{ for $1\leq j<n$}.
$$
By the equality (\ref{degh=1-lab4}) and Lemma \ref{ba1}(iv) it immediately follows that $\omega (\varphi _0)= n\omega (b_0)$, which is impossible.

Now assume that $\omega$ is a Dumas valuation of the first kind. Of course again we have (\ref{degh=1-lab1}). Let us here first prove Claim 2. Similarly as in the proof of Claim 1 above we deduce that
$
\omega (a_{n-2}) \leq \omega (\varphi _0)/n;
$
cf. (\ref{degh=1-lab2}). Next again by inductive argument we obtain that
$
\omega (a_{n-j}) \leq (j-1)\omega (\varphi _0)/n;
$
cf. (\ref{degh=1-lab3}). The concluding argument is the same one as above for Claim 1.

As the second step we prove Claim 1. Now again we have the equality (\ref{degh=1-lab4}), and also for its typical summand $\sigma _j$ the equality (\ref{degh=1-lab5}). We claim this time to have the inequalities $\omega (\sigma _j)<n\omega (b_0)$, for every $1\leq j<n$. Indeed, by the assumption ({\sf B1}) of the Conjecture and the inequality $\omega (\varphi _0) <n\omega (b_0)$ we obtain at once that $\omega (\varphi _j) <n\omega (b_0)/(j+1)$. It remains to observe that for every $j$ we have that $n/(j+1) \leq n-j$. The concluding argument is exactly the same one as for the above proof of Claim 2.
\end{proof}

Now we formulate the main result of this section.

\begin{thm}
\label{THMDumas1kind}
Let $R$ be an integral domain, $n>1$ a natural number and elements $\varphi _0,\varphi _1,\ldots ,\varphi _n\in R$. Suppose $\omega$ is a Dumas valuation of the first kind, defined on $R$ and satisfying the conditions {\sf (A)} and {\sf (B1)} of the Conjecture. Then for every $n\leq 7$ the polynomial $f=\sum _{j=0}^n\varphi _jX^j$ cannot be written as a product of two nonconstant polynomials from $R[X]$.
\end{thm}

Our first decisive step for the proof of the stated theorem will be performed via the following lemma. Note that in Lemma \ref{degh=1-lemma} we have proved that for $n\geq 2$ the case $(s,t)=(n-1,1)$ is impossible. In the lemma given below we will show that for $n\geq 4$ the case $(s,t)=(n-2,2)$ is impossible as well. Perhaps it might be quite surprising that the later case is much more difficult to treat.

\begin{rmk}
Related to our computations in what follows note the following obvious fact: If $1\leq \lambda < n$ and $(\lambda /n)\omega (\varphi _0)\in\mathbb N$, then $\gcd (n,\omega (\varphi _0))\neq 1$. Thus in particular for a Dumas valuation $\omega$ on $R$ satisfying ({\sf A}) of the Conjecture, in ({\sf B1}) and ({\sf B2}) we have in fact strict inequalities. For the same reason in various cases we consider in our discussions below it will be just strict inequalities; see, e.g., the {\sf Case P1} to {\sf Case P3} in the proof of the next lemma.
\end{rmk}

\begin{lem}
\label{degh=2-lemma}
Let $\omega$ be a Dumas valuation of the first kind, defined on an integral domain $R$. Let $f\in R[X]$ be a polynomial as in the Conjecture, of degree $n\geq 4$. Then there are no polynomials $g,h\in R[X]$ such that $f=gh$ and $\deg h=2$.
\end{lem}

\begin{proof}
Suppose to the contrary. Then we have $g=a_{n-2}X^{n-2}+\cdots +a_1X+a_0$ and $h=b_2X^2+b_1X+b_0$, where by Lemma \ref{obvious lemma} we in particular have that $\omega (a_{n-2})=0=\omega (b_2)$. And thus we have the following polynomial coefficients equalities:
 \begin{equation*}
  \begin{split}
    a_1b_0+a_0b_1 & =\varphi _1 \qquad\qquad\qquad\qquad\qquad\qquad   (E\! :\! 1) \\
    a_2b_0 + a_1b_1 +a_0b_2 & =\varphi _2 \qquad\qquad\qquad\qquad\qquad\qquad   (E\! :\! 2) \\
    a_3b_0 +a_2b_1 +a_1b_2 & =\varphi _3 \qquad\qquad\qquad\qquad\qquad\qquad   (E\! :\! 3) \\
       &  \cdots \qquad\qquad\qquad  \qquad\qquad\qquad\quad\;\; \cdots \\
   a_{n-3}b_0 +a_{n-4}b_1+ a_{n-5}b_2 & =\varphi _{n-3} \qquad\qquad\qquad\qquad\qquad\quad   (E\! :\! \text{n\!\! --\! 3}) \\
   a_{n-2}b_0 +a_{n-3}b_1+ a_{n-4}b_2 & =\varphi _{n-2} \qquad\qquad\qquad\qquad\qquad\quad   (E\! :\! \text{n\!\! --\! 2}) \\
   a_{n-2}b_1 +a_{n-3}b_2 & =\varphi _{n-1} \qquad\qquad\qquad\qquad\qquad\quad   (E\! :\! \text{n\!\! --\! 1})
   \end{split}
  \end{equation*}
Our argument will distinguish the following three possibilities:
\begin{itemize}
 \item[{\sf Case}]{\sf P1}. $\omega (a_0) > (n-2)\omega (\varphi _0)/n$.
 \item[{\sf Case}]{\sf P2}. $\omega (a_0) < 2\omega (\varphi _0)/n$.
 \item[{\sf Case}]{\sf P3}. $2\omega (\varphi _0)/n < \omega (a_0)< (n-2)\omega (\varphi _0)/n$.
\end{itemize}

\vspace{0.5ex}

Let us emphasize that in what follows we always bear in mind that for our polynomial $f$ both of the conditions {\sf (A)} and {\sf (B1)} do hold. Also for $\omega$ we have fulfilled the conditions (D1), (D2-1) and (D3-1). Furthermore, we freely use Lemma \ref{ba1}, mostly without explicit mentioning.

$\bullet$ We begin by proving that the {\sf Case P1} is impossible. In order to do that suppose to the contrary, and note that now $\omega (b_0) < 2\omega (\varphi _0)/n$. We refine our argumentation by considering two subcases: {\sf Case S1} and {\sf Case S2}.

\begin{itemize}
 \item[{\sf Case}]{\sf S1}. $\omega (b_1)<\omega (\varphi _0)/n$.
\end{itemize}

\noindent
Then by (E\! :\! \text{n\!\! --\! 1}) we deduce that necessarily $\omega (a_{n-3})< \omega (\varphi _0)/n$. As now we have that both $\omega (a_{n-3}b_1)$ and $\omega (a_{n-2}b_0)$ are less than $2\omega (\varphi _0)/n$, by (E\! :\! \text{n\!\! --\! 2}) it is clear that we must have $\omega (a_{n-4})< 2\omega (\varphi _0)/n$. By inductive argument one can see at once that
\begin{equation}
\label{degh=2-lemma-label0}
\omega (a_j) < \frac{n-j-2}{n}\omega (\varphi _0), \quad \text{ for $1\leq j\leq n-3$}.
\end{equation}
Hence  it follows that both $\omega (a_2b_0)$ and $\omega (a_1b_1)$ are less than $(n-2)\omega (\varphi _0)/n$. Using Lemma \ref{ba1}(iv) by (E\! :\! 2) we conclude that $\omega (a_0)<(n-2) \omega (\varphi _0)$; a contradiction.

\begin{itemize}
 \item[{\sf Case}]{\sf S2}. $\omega (b_1)>\omega (\varphi _0)/n$.
\end{itemize}

\noindent
Now by (E\! :\! \text{n\!\! --\! 1}) it is clear that necessarily
\begin{equation}
 \label{degh=2-lemma-label1}
\omega (a_{n-3})= \omega (b_1),
\end{equation}
and therefore that $\omega (a_{n-3}b_1)>2\omega (\varphi _0)/n$. Hence by (E\! :\! \text{n\!\! --\! 2}) it is immediate that
\begin{equation}
 \label{degh=2-lemma-label2}
\omega (a_{n-4})=\omega (a_{n-3}b_1)=2\omega (b_1).
\end{equation}
Next observe that by (\ref{degh=2-lemma-label1}) and the last equality we have that $\omega (a_{n-3}b_0)< \omega (a_{n-4}b_1)$. Using this, by (E\! :\! \text{n\!\! --\! 3}) it follows at once that necessarily
\begin{equation}
 \label{degh=2-lemma-label3}
\omega =\omega (a_{n-4}b_1)=3\omega (b_1).
\end{equation}
Proceeding by induction we see that
\begin{equation}
 \label{degh=2-lemma-label4}
\omega (a_j)= (n-j-2)\omega (b_1), \quad \text{ for $0\leq j\leq n-3$}.
\end{equation}
As now we have that $\omega (b_0)< 2\omega (b_1)$, using the above equalities for $j=0,1$ it is easy to deduce that $\omega (a_1b_0)< \omega (a_0b_1)$, and therefore by (E\! :\! 1) we conclude that
$$
\omega (\varphi _1) =\omega (a_0b_1)= (n-1)\omega (b_1) > (n-1)\omega (\varphi _0)/n;
$$
a contradiction to {\sf (B1)}.

\vspace{0.5ex}

$\bullet$ As the next step we will prove a more complicated claim, that the {\sf Case P2} is impossible. Again we suppose to the contrary, and note that then for the coefficient $b_0$ the inequality $\omega (b_0) > (n-2)\omega (\varphi _0)/n$ holds. And here we consider two subcases.

\begin{itemize}
 \item[{\sf Case}]{\sf P2\! -1}. $\omega (a_1)>\omega (\varphi _0)/n$.
\end{itemize}

\noindent
Then we have that $\omega (a_1b_0) > (n-1)\omega (\varphi _0)/n$ and therefore by (E\! :\! 1) it follows that necessarily $\omega (a_1b_0)= \omega (a_0b_1)$. But as we have $\omega (b_0)-\omega (a_0) > (n-4)\omega (\varphi _0)/n$, by the last equality we deduce that
\begin{equation}
 \label{degh=2-lemma-label6}
\omega (b_1) > \Bigl( \frac{n-4}{n} +\frac{1}{n}\Bigr) \omega (\varphi _0)= \frac{n-3}{n}\omega (\varphi _0).
\end{equation}
Hence further observe that again (\ref{degh=2-lemma-label1}) holds. Also note that here we have the inequality
\begin{equation}
 \label{degh=2-lemma-label7}
\omega (b_0) -\omega (b_1)= \omega (a_0) -\omega (a_1) < \omega (\varphi _0)/n.
\end{equation}
But as we have $n\geq 4$, (\ref{degh=2-lemma-label6}) implies that in particular $\omega (b_1) > \omega (\varphi _0)/n$. Now by (\ref{degh=2-lemma-label1}) we deduce that $\omega (b_0)<\omega (a_{n-3}b_1)=2\omega (b_1)$. As we have, again by (\ref{degh=2-lemma-label6}), that
$$
\omega (a_{n-3}b_1)> 2(n-3)\omega (\varphi _0)/n \geq 2\omega (\varphi _0)/n,
$$
by (E\! :\! \text{n\!\! --\! 2}) we finally obtain that necessarily again (\ref{degh=2-lemma-label2}) holds. Similarly as for the {\sf Case S2} by inductive argument we prove that again (\ref{degh=2-lemma-label4}) holds. And then in particular $\omega (a_0)= (n-2) \omega (b_1)$, as there. By using one more time that $n\geq 4$ we conclude that
$$
\omega (a_0) > (n-2)\omega (\varphi _0)/n \geq 2\omega (\varphi _0)/n;
$$
a contradiction.

\begin{itemize}
 \item[{\sf Case}]{\sf P2\! -2}. $\omega (a_1)<\omega (\varphi _0)/n$.
\end{itemize}

\noindent
Here first note that necessarily $a_2\neq 0$. To se this, suppose to the contrary. Then (E\! :\! 3) becomes $a_3b_0+a_1b_2=\varphi _3$. We claim that $a_3$ must be zero as well. Namely, otherwise we would have that $\omega (a_3b_0)\geq \omega (b_0)$, and therefore
$$
\omega (\varphi _3)=\omega (a_3b_0+a_1b_2)=\omega (a_3b_0) > (n-2)\omega (\varphi _0)/n;
$$
a contradiction to {\sf (B1)}. By inductive argument we obtain that $a_j=0$ for every $3\leq j\leq n-3$. Thus in particular (E\! :\! \text{n\!\! --\! 2}) becomes an equality $a_{n-2}b_0=\varphi _{n-2}$. And therefore
$$
\omega (\varphi _{n-2})=\omega (b_0) > (n-2)\omega (\varphi _0)/n\geq 2\omega (\varphi _0)/n;
$$
again a contradiction.

Now when we know that $a_2$ is nonzero, by Lemma \ref{ba1}(iii) it follows that $\omega (a_2b_0) \geq \omega (b_0)$. As we also have that
$$
\omega (a_0)< 2\omega (\varphi _0)/n \leq (n-2)\omega (\varphi _0)/n <\omega (b_0),
$$
by (E\! :\! 2) we conclude that necessarily
\begin{equation}
 \label{degh=2-lemma-label8}
\omega (a_2b_0) =\omega (a_1b_1).
\end{equation}
And hence we have that
\begin{equation}
 \label{degh=2-lemma-label9}
\omega (b_1) =\omega (a_2b_0)- \omega (a_1) > \Bigl( \frac{n-2}{n}-\frac{1}{n}\Bigr) \omega (\varphi _0)= \frac{n-3}{n}\omega (\varphi _0);
\end{equation}
cf. (\ref{degh=2-lemma-label6}). As a consequence, (\ref{degh=2-lemma-label1}) holds again.

At this point we distinguish two more sub-subcases:
\begin{itemize}
 \item[{\sf Case}]{\sf P2\!\! -2\!\! -1}. $\omega (b_0) > (n-1)\omega (\varphi _0)/n$.
 \item[{\sf Case}]{\sf P2\!\! -2\!\! -2}. $\omega (b_0) < (n-1)\omega (\varphi _0)/n$.
\end{itemize}

\noindent
Suppose the first sub-subcase. Then in the same way as for (\ref{degh=2-lemma-label9}) we have that now $\omega (b_1) > (n-2)\omega (\varphi _0)/n$. And therefore using (\ref{degh=2-lemma-label1}) and the fact $n\geq 4$ it follows that
$$
\omega (a_{n-3}b_1) = 2\omega (b_1)> (2n-4)\omega (\varphi _0)/n \geq \omega (\varphi _0)\geq  \omega (b_0).
$$
Hence by (E\! :\! \text{n\!\! --\! 2}) we deduce that here (\ref{degh=2-lemma-label2}) holds as well. By the same reasoning as for the {\sf Case S2} we obtain a contradiction to the condition {\sf (B1)}.

Next suppose the {\sf Case P2-2-2}. Then we again have (\ref{degh=2-lemma-label9}) and therefore using (\ref{degh=2-lemma-label1}) for $n\geq 5$ we obtain that
$$
\omega (a_{n-3}b_1) = 2\omega (b_1)> (2n-6)\omega (\varphi _0)/n \geq (n-1)\omega (\varphi _0)/n >  \omega (b_0).
$$
The rest of the argument is exactly the same one as for the first sub-subcase. Here the case $n=4$ must be treated separately. The details are left to the reader.

\vspace{0.5ex}

$\bullet$ It remains to prove the most complicated claim, that the {\sf Case P3} is impossible. First note that now we also have for the polynomial coefficient $b_0$ the inequalities
\begin{equation}
 \label{degh=2-lemma-label10}
2\omega (\varphi _0)/n < \omega (b_0) < (n-2)\omega (\varphi _0)/n.
\end{equation}
Again we will separately treat the same two subcases as we did in {\sf Case P1}.

First consider the {\sf Case S2}, which is more involved. As before, here we have the equality (\ref{degh=2-lemma-label1}). Let us treat the following three mutually exclusive sub-subcases:

\begin{itemize}
 \item[{\sf Case}]{\sf S2\!\! -1}. $\omega (a_{n-4}) < \omega (b_0)$.
 \item[{\sf Case}]{\sf S2\!\! -2}. $\omega (a_{n-4}) > \omega (b_0)$.
 \item[{\sf Case}]{\sf S2\!\! -3}. $\omega (a_{n-4}) = \omega (b_0)$.
\end{itemize}

\vspace{0.5ex}

\noindent
Suppose we have the first one sub-subcase. Then, taking into account (\ref{degh=2-lemma-label1}) and (\ref{degh=2-lemma-label10}), by the equality (E\! :\! \text{n\!\! --\! 2}) it follows that
\begin{equation}
 \label{degh=2-lemma-label11}
\omega (b_0)= \omega (a_{n-3}b_1)= 2\omega (b_1).
\end{equation}
Further as we are in the {\sf Case S2-1} it is clear that
\begin{equation}
 \label{degh=2-lemma-label12}
\omega (a_{n-3}b_0) > \omega (a_{n-4}b_1).
\end{equation}
As we have that
$$
\omega (a_{n-3}b_0)= \omega (b_1)+ \omega (b_0)= 3\omega (b_1) > 3\omega (\varphi _0)/n,
$$
by (\ref{degh=2-lemma-label12}) and the equality (E\! :\! \text{n\!\! --\! 3}) it is necessarily
\begin{equation}
 \label{degh=2-lemma-label12A}
\omega (a_{n-5})= \omega (a_{n-3}b_0)= 3\omega (b_1).
\end{equation}
Proceeding inductively one can see that
\begin{equation}
 \label{degh=2-lemma-label13}
\omega (a_{n-k-2}) =k \omega (b_1),
\end{equation}
for $k>2$; which is in fact the equality in (\ref{degh=2-lemma-label4}), written in an equivalent way. (Note that for $k=2$ we get $\omega (a_{n-4}) < \omega (b_0)=2\omega (b_1)$.) In particular for $k=n-5$ and $k=n-4$ we have $\omega (a_3)=(n-5)\omega (b_1)$ and $\omega (a_2)=(n-4)\omega (b_1)$. So by (\ref{degh=2-lemma-label11}) we also have that
$$
\omega (a_3b_0) =(n-3)\omega (b_1)=\omega (a_2b_1).
$$
Now it is plain that by (E\! :\! 3) we have the inequality
\begin{equation}
 \label{degh=2-lemma-label14}
\omega (a_1)\leq (n-3)\omega (b_1).
\end{equation}
Suppose that moreover the last inequality is strict. Then we would have that $\omega (a_1b_1)< (n-2)\omega (b_1)$. And at the same time by (\ref{degh=2-lemma-label11}) and (\ref{degh=2-lemma-label13}) it would follow that
$$
\omega (a_2b_0)= \omega (a_2)+ \omega (b_0)= \bigl( (n-4)+2\bigr) \omega (b_1) =(n-2)\omega (b_1).
$$
Thus we have the inequalities
$$
\omega (a_0) < \frac{n-2}{n}\omega (\varphi _0) <\omega (a_2b_0).
$$
By this applied to (E\! :\! 2) we would in particular have that
$$
\omega (\varphi _2)= \omega (a_2b_0) > (n-2)\omega (\varphi _0)/n;
$$
a contradiction. As a consequence, we see that in (\ref{degh=2-lemma-label14}) the equality holds. Finally it is easy to check that we have $\omega (a_1b_0) > \omega (a_0b_1)$. Thus by (E\! :\! 1) we conclude that
$$
\omega (\varphi _1) =\omega (a_1b_0) =(n-1)\omega (b_1) > (n-1) \omega (\varphi _0)/n;
$$
a contradiction.

Now consider the {\sf Case S2-2}. A similar argument as in the {\sf Case S2}  for {\sf Case P1} gives that here (\ref{degh=2-lemma-label2}) holds as well. Further we have the inequality $\omega (a_{n-3}b_0) < \omega (a_{n-4}b_1)$; cf. (\ref{degh=2-lemma-label12}). And hence by (E\! :\! \text{n\!\! --\! 3})
we obtain that
$$
\omega (a_{n-5})= \omega (a_{n-4}b_1)= 3\omega (b_1);
$$
cf. (\ref{degh=2-lemma-label12A}). In fact by inductive reasoning as we did for the {\sf Case S2-1} it follows that again (\ref{degh=2-lemma-label13}) holds; but this time for $k\geq 2$. The rest of the argument is the same one as there.

It remains to treat the {\sf Case S2-3}, which is more demanding. Now first observe that by (\ref{degh=2-lemma-label1}) the equality (E\! :\! \text{n\!\! --\! 2}) gives the inequality
\begin{equation}
 \label{degh=2-lemma-label15}
\omega (a_{n-3}b_1)=2\omega (b_1)\leq \omega (b_0).
\end{equation}
Indeed, by assuming the opposite and using the inequality of the {\sf Case S2}, it would follow that $$
\omega (\varphi _{n-2}) = \omega (a_{n-3}b_1)> 2\omega (\varphi _0)/n;
$$
a contradiction. As we already did that several times before, here for the {\sf Case S2-3} we also need to refine our analysis in a way that we will consider two subcases for (\ref{degh=2-lemma-label15}); that will be called {\sf Case T1} and {\sf Case T2}.

\begin{itemize}
 \item[{\sf Case}]{\sf T1}. $2\omega (b_1)<\omega (b_0)$.
\end{itemize}

\noindent
By (\ref{degh=2-lemma-label1}) we have
$$
\omega (a_{n-3}b_0)= \omega (b_1)+ \omega (a_{n-4})= \omega (a_{n-4}b_1).
$$
And hence by (E\! :\! \text{n\!\! --\! 3}) it follows that
\begin{equation}
 \label{degh=2-lemma-label16}
\omega (a_{n-5}) \leq \omega (a_{n-4}b_1)= \omega (b_0)+ \omega (b_1).
\end{equation}
Next we consider (E\! :\! \text{n\!\! --\! 4}), and note that
$$
\omega (a_{n-5}b_1)\leq \omega (b_0)+ 2\omega (b_1)< 2\omega (b_0)= \omega (a_{n-4}b_0).
$$
As we have that $\omega (a_{n-4}b_0)=2\omega (b_0) > 4\omega (\varphi _0)/n$, it is immediate that necessarily
\begin{equation}
\label{degh=2-lemma-label17}
 \omega (a_{n-6})= \omega (a_{n-4}b_0)= 2\omega (b_0).
\end{equation}
Let us now state the following auxiliary claim.
\begin{claim1}
For every $k\leq \bigl\lceil \frac{n-3}{2}\big\rceil$ we have the following:
 \begin{equation*}
  \begin{split}
  \omega (a_{n-2k-2}) & = k\omega (b_0) ; \\
  \omega (a_{n-2k-3}) & \leq k\omega (b_0) + \omega (b_1).
  \end{split}
 \end{equation*}
\end{claim1}
In order to prove this Claim we proceed by induction, where the base of the induction, for $k=1$, is given by the condition {\sf Case S2-3} itself and the inequality (\ref{degh=2-lemma-label16}). The details of the inductive step are left to the reader.

Now we use the above Claim for concluding our argument. First suppose $n=2m$; i.e., $n$ is even. Then for $k=m-2$ we get
$$
\omega (a_2)= (m-2)\omega (b_0) \quad \text{ and } \quad \omega (a_1)\leq (m-2)\omega (b_0) +\omega (b_1).
$$
As we are in the {\sf Case T1}, it follows that
$$
\omega (a_2b_0) > (2m-2)\omega (b_1) > (n-2)\omega (\varphi _0)/n.
$$
We also have that $\omega (a_2b_0) > \omega (a_1b_1)$. Thus by (E\! :\! 2) we deduce that necessarily
$$
\omega (a_0) =\omega (a_2b_0) > (n-2)\omega (\varphi _0)/n;
$$
a contradiction.

Next suppose that $n=2m-1$; i.e., $n$ is odd. Then for $k=(n-3)/2$ we get
$$
\omega (a_1)= (n-3)\omega (b_0)/2 \quad \text{ and } \quad \omega (a_2)\leq (n-5)\omega (b_0)/2 +\omega (b_1).
$$
Further observe that now $\omega (a_2b_0) \leq \omega (a_1b_1)$. And so by (E\! :\! 2) we obtain that
$$
\omega (a_0) \leq \omega (a_1b_1) =(n-3)\omega (b_0)/2 +\omega (b_1).
$$
But then as we are still in the {\sf Case T1} it follows that
$$
\omega (a_1b_0) > (n-3)\omega (b_0)/2 +2\omega (b_1) \geq \omega (a_0b_1).
$$
Finally, by (E\! :\! 1) we deduce that
$$
\omega (\varphi _1)= \omega (a_1b_0) = \frac{n-1}{2}\omega (b_0) > \frac{n-1}{n}\omega (\varphi _0);
$$
again a contradiction.

\begin{itemize}
 \item[{\sf Case}]{\sf T2}. $2\omega (b_1)=\omega (b_0)$.
\end{itemize}

\noindent
For the ease of following let us first perform some calculations. At the beginning note that we again have (\ref{degh=2-lemma-label16}). Now if we
have the strict inequality there, then by  (E\! :\! \text{n\!\! --\! 4}) it follows that  (\ref{degh=2-lemma-label17}) holds as well. On the other hand if $\omega (a_{n-5})=\omega (b_0) + \omega (b_1)$, then in particular $\omega (a_{n-5}b_1)=\omega (a_{n-4}b_0)$. And then by  (E\! :\! \text{n\!\! --\! 4}) we deduce that
\begin{equation}
\label{degh=2-lemma-label18}
 \omega (a_{n-6})\leq 2\omega (b_0).
\end{equation}
Next consider the equality (E\! :\! \text{n\!\! --\! 5}). For the strict inequality in  (\ref{degh=2-lemma-label16}) it is immediate that then $\omega (a_{n-5}b_0)< \omega (a_{n-6}b_1)$, and therefore that necessarily
$$
\omega (a_{n-7})=\omega (a_{n-6}b_1)= 5\omega (b_1).
$$
On the other hand if we have the equality in (\ref{degh=2-lemma-label16}), it will follow that necessarily
\begin{equation}
\label{degh=2-lemma-label19}
 \omega (a_{n-7})\leq 5\omega (b_1).
\end{equation}
Further for the equality in  (\ref{degh=2-lemma-label16}) but the strict inequality in  (\ref{degh=2-lemma-label18}) we have that
$$
\omega (a_{n-6}b_1)< 2\omega (b_0)+ \omega (b_1)= \omega (a_{n-5}b_0).
$$
And so we deduce that in  (\ref{degh=2-lemma-label19}) is moreover the equality.

Now assume that we have proved that for some $k$ the following hold:
$$
\omega (a_{n-k})=(k-2)\omega (b_1)\quad \text{ and }\quad \omega (a_{n-k-1})<(k-1)\omega (b_1).
$$
Consider the equality (E\! :\! \text{n\!\! --\! k}).
As then we have that $\omega (a_{n-k-1}b_1) < \omega (a_{n-k}b_0)$, it follows that necessarily
$
\omega (a_{n-k-2})=\omega (a_{n-k}b_0)= k\omega (b_1).
$
As another possibility assume that for some $k$ we have proved that both
$$
\omega (a_{n-k}) < (k-2)\omega (b_1)\quad \text{ and }\quad \omega (a_{n-k-1}) = (k-1)\omega (b_1).
$$
Now we have the inequality $\omega (a_{n-k}b_0) < \omega (a_{n-k-1}b_1)$, and therefore necessarily $
\omega (a_{n-k-2})=\omega (a_{n-k-1}b_1)= k\omega (b_1).
$
As the third possibility assume we have both
$$
\omega (a_{n-k}) = (k-2)\omega (b_1)\quad \text{ and }\quad \omega (a_{n-k-1}) = (k-1)\omega (b_1).
$$
Here one can see at once that then
$$
\omega (a_{n-k-2}) \leq k\omega (b_1).
$$
As a conclusion of the above computations it is clear that the following claim hold.

\begin{claim2}
For every $k\leq n-1$ we have the inequality
\begin{equation}
\label{degh=2-lemma-label20}
 \omega (a_{n-k})\leq (k-2)\omega (b_1),
\end{equation}
and in particular there is no $j<n-1$ such that for both $k=j$ and $k=j+1$ we have the strict inequalities in \textup{(\ref{degh=2-lemma-label20})}.
\end{claim2}

Thus by the last claim we have that in particular it must necessarily hold
\begin{equation}
\label{degh=2-lemma-label21}
 \omega (a_2)=(n-4)\omega (b_1) \quad \text{ or } \quad \omega (a_1)=(n-3)\omega (b_1).
\end{equation}
Suppose the latter equality. Then
$$
\omega (a_1b_0) =(n-1)\omega (b_1) > (n-1) \omega (\varphi _0)/n,
$$
and then by (E\! :\! 1) it must be that $\omega (a_1b_0)= \omega (a_0b_1)$. Hence it follows that
$$
\omega (a_0)=\omega (a_1b_0) -\omega (b_1) > \frac{n-2}{n}\omega (\varphi _0);
$$
a contradiction. Next suppose we have the first equality in (\ref{degh=2-lemma-label21}), but at the same time that $\omega (a_1) < (n-3)\omega (b_1)$. Then $\omega (a_2b_0)=(n-2)\omega (b_1) > \omega (a_1b_1)$, and so by (E\! :\! 2) we obtain that
$$
\omega (a_0)=\omega (a_2b_0) > \frac{n-2}{n}\omega (\varphi _0);
$$
again a contradiction.

What remains to do in order to finish the {\sf Case P3} is to consider its subcase {\sf Case S1}. (Note here that the argument given below is in fact very easy; but it seems there is no such ``simple trick'' while dealing with the {\sf Case S2}.) First note the inequality $\omega (a_0b_1) < (n-1)\omega (\varphi _0)/n$, and then by (E\! :\! 1) it follows that $\omega (a_1b_0) \leq (n-1) \omega (\varphi _0)/n$ as well. Hence,
$$
\omega (a_1) \leq \frac{n-1}{n} \omega (\varphi _0) -\omega (b_0) < \frac{n-3}{n} \omega (\varphi _0).
$$
As now we also have that $\omega (a_1b_1) < (n-2)\omega (\varphi _0)/n$, by (E\! :\! 2) we conclude that necessarily $\omega (a_2b_0) \leq (n-2)\omega (\varphi _0)/n$, and therefore $\omega (a_2) < (n-4)\omega (\varphi _0)/n$. By induction it is immediate that again (\ref{degh=2-lemma-label0}) holds.
In particular we have that
$$
\omega (a_{n-4}) < 2\omega (\varphi _0)/n \quad \text{ and } \quad \omega (a_{n-3}) < \omega (\varphi _0)/n.
$$
And so both $\omega (a_{n-4}b_2)$ and $\omega (a_{n-3}b_1)$ are less than $2\omega (\varphi _0)/n$, while $\omega (a_{n-2}b_0) > 2\omega (\varphi _0)/n$. As a conclusion,
$
\omega (\varphi _2)= \omega (b_0) > 2\omega (\varphi _0)/n;
$
a contradiction. This finishes the proof of our lemma.
\end{proof}

Of course we in particular claim that for $n\geq 6$ the case $(s,t)=(n-3,3)$ is also impossible; but we do not have a proof for that. The best we can do is to prove the following lemma which is our second decisive step for the proof of Theorem \ref{THMDumas1kind}. Note that our argument for the lemma is surprisingly involved.

\begin{lem}
\label{degh=3-lemma}
Let $\omega$ be a Dumas valuation of the first kind, defined on an integral domain $R$. Let $f\in R[X]$ be a polynomial as in the Conjecture, of degree $n=6$ or $7$. Then there are no polynomials $g,h\in R[X]$ such that $f=gh$ and $\deg h=3$.
\end{lem}

\begin{proof}
We will only treat the case $n=7$, which is considerably more tedious than the case $n=6$ which will be left to the reader; cf. the proof of Lemma \ref{2Ddegh=3-lemma}. Our arguing is similar to the one of the previous lemma, but with more technical details and obstacles along the way. Note that some of our notation is exactly the same one as there; but this will make no confusion.

Here we have $g=a_4X^4+\cdots +a_1X+a_0$ and $h=b_3X^3+b_2X^2+b_1X+b_0$, where $\omega (a_4)=0=\omega (b_3)$. Then we have the following polynomial coefficients equalities:
 \begin{equation*}
  \begin{split}
    a_1b_0+a_0b_1 & =\varphi _1 \qquad\qquad\qquad\qquad\qquad\qquad   (E\! :\! 1) \\
    a_2b_0 + a_1b_1 +a_0b_2 & =\varphi _2 \qquad\qquad\qquad\qquad\qquad\qquad   (E\! :\! 2) \\
    a_3b_0 +a_2b_1 +a_1b_2 + a_0b_3 & =\varphi _3 \qquad\qquad\qquad\qquad\qquad\qquad  (E\! :\! 3) \\
    a_4b_0 +a_3b_1 +a_2b_2 + a_1b_3 & =\varphi _4 \qquad\qquad\qquad\qquad\qquad\qquad  (E\! :\! 4) \\
    a_4b_1 +a_3b_2 +a_2b_3 & =\varphi _5 \qquad\qquad\qquad\qquad\qquad\qquad  (E\! :\! 5) \\
    a_4b_2 +a_3b_3 & =\varphi _6 \qquad\qquad\qquad\qquad\qquad\qquad  (E\! :\! 6)
   \end{split}
  \end{equation*}
Let us distinguish the following three possibilities:
\begin{itemize}
 \item[{\sf Case}]{\sf P1}. $\omega (a_0) > 4\omega (\varphi _0)/7$.
 \item[{\sf Case}]{\sf P2}. $\omega (a_0) < 3\omega (\varphi _0)/7$.
 \item[{\sf Case}]{\sf P3}. $3\omega (\varphi _0)/7 < \omega (a_0)< 4\omega (\varphi _0)/7$.
\end{itemize}

\vspace{0.5ex}

Again recall that for our polynomial $f$ both of the conditions {\sf (A)} and {\sf (B1)} do hold. Also for $\omega$ the conditions (D1), (D2-1) and (D3-1) are fulfilled. And again the part (iv) of Lemma \ref{ba1} is at our hands all the time.

$\bullet$ We first prove that the {\sf Case P1} is impossible. So suppose to the contrary, and note that now $\omega (b_0) < 3\omega (\varphi _0)/7$. We refine our discussion by considering further two possibilities:
\begin{itemize}
 \item[{\sf Case}]{\sf S1}. $\omega (b_2) < \omega (\varphi _0)/7$.
 \item[{\sf Case}]{\sf S2}. $\omega (b_2) > \omega (\varphi _0)/7$.
\end{itemize}

Assume the {\sf Case S1}. First note that by (E\! :\! 6) we must have that
\begin{equation}
\label{degh=3-lemma-label1}
\omega (a_3)< \omega (\varphi _0)/7.
\end{equation}
Then consider the subcase
\begin{itemize}
 \item[{\sf Case}]{\sf S1\!\! -1}. $\omega (a_2) < 2\omega (\varphi _0)/7$.
\end{itemize}
As we have $\omega (a_3b_2)< 2\omega (\varphi _0)/7$, by (E\! :\! 5) it is immediate that
\begin{equation}
\label{degh=3-lemma-label2}
\omega (b_1)< 2\omega (\varphi _0)/7.
\end{equation}
Now it is clear that each of the values $\omega (a_4b_0)$, $\omega (a_3b_1)$ and $\omega (a_2b_2)$ is less than $3\omega (\varphi _0)/7$. And therefore by (E\! :\! 4) we conclude that $\omega (a_1)< 3\omega (\varphi _0)/7$ as well. Taking into account (\ref{degh=3-lemma-label1}) and (\ref{degh=3-lemma-label2}) it follows that each of the values $\omega (a_3b_0)$, $\omega (a_2b_1)$ and $\omega (a_1b_2)$ is less than $4\omega (\varphi _0)/7$. Hence by (E\! :\! 3) it follows that
\begin{equation}
\label{degh=3-lemma-label2A}
\omega (\varphi _3)=\omega (a_0b_3)>4 \omega (\varphi _0)/7;
\end{equation}
a contradiction. Next  consider the subcase
\begin{itemize}
 \item[{\sf Case}]{\sf S1\!\! -2}. $\omega (a_2) > 2\omega (\varphi _0)/7$.
\end{itemize}
This time by (E\! :\! 5) we have that
\begin{equation}
\label{degh=3-lemma-label3}
\omega (b_1)= \omega (a_2).
\end{equation}
Hence $\omega (a_0b_1)> 6\omega (\varphi _0)/7$ and then by (E\! :\! 1) it is necessarily that
\begin{equation}
\label{degh=3-lemma-label3B}
\omega (a_1b_0)= \omega (a_0b_1)
\end{equation}
holds. In particular note here that
\begin{equation}
\label{degh=3-lemma-label3A}
\omega (a_1)= \omega (a_0b_1)-\omega (b_0)> 3\omega (\varphi _0)/7.
\end{equation}
At this point we consider the following three sub-subcases:
\begin{itemize}
 \item[{\sf Case}]{\sf U1}. $\omega (a_3) < \omega (b_2)$.
 \item[{\sf Case}]{\sf U2}. $\omega (a_3) > \omega (b_2)$.
 \item[{\sf Case}]{\sf U3}. $\omega (a_3) = \omega (b_2)$.
\end{itemize}

Assume the {\sf Case U1}. Then by (\ref{degh=3-lemma-label3}) we have the inequality
$\omega (a_3b_1)< \omega (a_2b_2)$,
and therefore by (E\! :\! 4) it follows that
\begin{equation}
\label{degh=3-lemma-label4}
\omega (a_1)= \omega (a_2b_2).
\end{equation}
As we clearly have that $\omega (a_2) > 2\omega (b_2)$, by (\ref{degh=3-lemma-label3}) again and (\ref{degh=3-lemma-label4}) it is immediate that
\begin{equation}
\label{degh=3-lemma-label5}
\omega (a_2b_1)> \omega (a_1b_2).
\end{equation}
Next observe that $\omega (a_3b_0)< 4\omega (\varphi _0)/7$ and thus by (E\! :\! 3) we deduce that \begin{equation}
\label{degh=3-lemma-label6}
\omega (a_0)= \omega (a_2b_1)=2\omega (a_2).
\end{equation}
Now by (\ref{degh=3-lemma-label3}), (\ref{degh=3-lemma-label4}) and (\ref{degh=3-lemma-label6}) we see that the equality (\ref{degh=3-lemma-label3B}) is equivalent to the equality $\omega (b_2b_0)=2\omega (a_2)$. But the left-hand side of the later equality is less than $4\omega (\varphi _0)/7$, while the right-hand side is greater than this number; a contradiction.

Next assume the {\sf Case U2}. Then
\begin{equation}
\label{degh=3-lemma-label6A}
\omega (a_2b_2) < \omega (a_3b_1)
\end{equation}
holds and therefore by (E\! :\! 4) we obtain the equality
\begin{equation}
\label{degh=3-lemma-label7}
\omega (a_1)= \omega (a_3b_1).
\end{equation}
Using this we obtain that (\ref{degh=3-lemma-label5}) again holds. And (\ref{degh=3-lemma-label6}) holds as well. Finally using (\ref{degh=3-lemma-label7}) we deduce that this time the equality (\ref{degh=3-lemma-label3B}) is equivalent to $\omega (a_3b_0)=2\omega (a_2)$; which is one more time impossible.

It remains to assume the {\sf Case U3}. Here we in particular have that $\omega (a_3b_1)=\omega (a_2b_2)$ and thus by (\ref{degh=3-lemma-label3A}) and (E\! :\! 4) it follows that again both (\ref{degh=3-lemma-label4}) and (\ref{degh=3-lemma-label5}) do hold. The rest of the argument is exactly the same one as for the {\sf Case U1}.

\vspace{0.5ex}

Now assume the {\sf Case S2}. This time by (E\! :\! 6) it is immediate that
\begin{equation}
\label{degh=3-lemma-label8}
\omega (a_3)= \omega (b_2).
\end{equation}
We will consider the same two subcases as for the {\sf Case S1}. First assume that we have {\sf Case S1-1}. As here $\omega (a_3b_2)> 2\omega (\varphi _0)/7$ by (E\! :\! 5) we see that necessarily
\begin{equation}
\label{degh=3-lemma-label9}
\omega (b_1)= \omega (a_3b_2)=2\omega (b_2).
\end{equation}
Further observe that we have (\ref{degh=3-lemma-label6A}) and also that $\omega (a_3b_1)=3\omega (b_2)> \omega (b_0)$. Therefore by (E\! :\! 4) we conclude that
\begin{equation}
\label{degh=3-lemma-label10}
\omega (a_1)= \omega (a_3b_1)=3\omega (b_2).
\end{equation}
Moreover by (\ref{degh=3-lemma-label8}) and the last equality we have the inequality
\begin{equation}
\label{degh=3-lemma-label10X}
\omega (a_3b_0)< \omega (a_1b_2).
\end{equation}
Also by (\ref{degh=3-lemma-label9}) and (\ref{degh=3-lemma-label10}) we obtain at once that
\begin{equation}
\label{degh=3-lemma-label10A}
\omega (a_2b_1) < \omega (a_1b_2).
\end{equation}
Hence, using that $\omega (a_1b_2)= 4\omega (b_2) > 4\omega (\varphi _0)/7$, by (E\! :\! 3) we deduce that
\begin{equation}
\label{degh=3-lemma-label11}
\omega (a_0)= \omega (a_1b_2)=4\omega (b_2).
\end{equation}
Finally observe that now
$
\omega (a_1b_0)< \omega (a_0b_1).
$
And therefore, as we have that $\omega (a_0b_1) > 6\omega (\varphi _0)/7$, it follows that
\begin{equation}
\label{degh=3-lemma-label11A}
\omega (\varphi _1) =\omega (a_0b_1) >6\omega (\varphi _0)/7;
\end{equation}
a contradiction.

Now assume the {\sf Case S1-2}. We refine our argumentation via three sub-subcases:
\begin{itemize}
 \item[{\sf Case}]{\sf V1}. $\omega (a_2) > \omega (b_1)$.
 \item[{\sf Case}]{\sf V2}. $\omega (a_2) < \omega (b_1)$.
 \item[{\sf Case}]{\sf V3}. $\omega (a_2) = \omega (b_1)$.
\end{itemize}

Assume first the {\sf Case V1}. Then by (\ref{degh=3-lemma-label8}) and (E\! :\! 5) it is clear that necessarily
\begin{equation}
\label{degh=3-lemma-label12}
\omega (a_2)= 2\omega (b_2).
\end{equation}
Further we have
\begin{equation}
\label{degh=3-lemma-label12A}
\omega (a_3b_1) =\omega (b_1) +\omega (b_2) <\omega (a_2b_2),
\end{equation}
and also
\begin{equation}
\label{degh=3-lemma-label13}
\omega (a_2b_2) = 3\omega (b_2) > 3\omega (\varphi _0)/7 > \omega (b_0).
\end{equation}
Therefore by (E\! :\! 4) we obtain the equality (\ref{degh=3-lemma-label4}); i.e., that
\begin{equation}
\label{degh=3-lemma-label14}
\omega (a_1) = 3\omega (b_2).
\end{equation}
Next, using the inequalities in (\ref{degh=3-lemma-label13}) and one more time that (\ref{degh=3-lemma-label8}) holds, we obtain that
$$
\omega (a_3b_0) = \omega (b_2b_0) < 4\omega (b_2)= \omega (a_1b_2).
$$
Also by (\ref{degh=3-lemma-label12}) we see that again $\omega (a_2b_1) < \omega (a_1b_2)$. As a conclusion, by (E\! :\! 3) we get (\ref{degh=3-lemma-label11}). Then by (\ref{degh=3-lemma-label14}) we have the inequality $\omega (a_1b_1)< \omega (a_0b_2)$. And by
(\ref{degh=3-lemma-label11}) and (\ref{degh=3-lemma-label12}) it follows that $\omega (a_2b_0)< \omega (a_0b_2)$. Taking into account the last two obtained inequalities we deduce that
$$
\omega (\varphi _2) =\omega (a_0b_2) =5\omega (b_2) > 5\omega (\varphi _0)/7;
$$
a contradiction.

As the second step assume the {\sf Case V2}. It is straightforward to check that now we have (\ref{degh=3-lemma-label9}), (\ref{degh=3-lemma-label10}) and (\ref{degh=3-lemma-label11}), where exactly the same argument as for the {\sf Case S1-1} works here as well.

Next assume the {\sf Case V3}. This time by (E\! :\! 5) it is immediate that we must have
\begin{equation}
\label{degh=3-lemma-label15}
\omega (a_2)= \omega (b_1) \geq \omega (a_3b_2)= 2\omega (b_2) > 2\omega (\varphi _0)/7.
\end{equation}
Hence $\omega (a_0b_1) > 6\omega (\varphi _0)/7$ and therefore we again have (\ref{degh=3-lemma-label3B}). Thus it follows that (\ref{degh=3-lemma-label3A}) holds too, and so we have the inequality
\begin{equation}
\label{degh=3-lemma-label16}
\omega (a_2b_0) < \omega (a_1b_1).
\end{equation}
Using this, as by (\ref{degh=3-lemma-label3A}) and (\ref{degh=3-lemma-label15}) we have that $\omega (a_1b_1) > 5\omega (\varphi _0)/7$, it is clear that (E\! :\! 2) gives the equality
\begin{equation}
\label{degh=3-lemma-label17}
\omega (a_0b_2)= \omega (a_1b_1).
\end{equation}
Further as we clearly have the inequality $\omega (b_0)<\omega (a_1)$, it necessarily follows the inequality (\ref{degh=3-lemma-label10X}). Next we claim that (\ref{degh=3-lemma-label10A}) holds as well. Namely, by (\ref{degh=3-lemma-label3B}) we see that this inequality is equivalent to the inequality
$$
2\omega (b_1) < \bigl( \omega (a_0) +\omega (b_1) -\omega (b_0)\bigr) +\omega (b_2),
$$
which is further equivalent to the inequality $\omega (b_1b_0)<\omega (a_0b_2)$. But the last one is via (\ref{degh=3-lemma-label17}) equivalent to (\ref{degh=3-lemma-label16}). Now, using that $\omega (a_1b_2) > 4\omega (\varphi _0)/7$, by (E\! :\! 3), (\ref{degh=3-lemma-label10X}) and (\ref{degh=3-lemma-label10A}) it is clear that $\omega (a_0)=\omega (a_1b_2)$; cf. (\ref{degh=3-lemma-label11}). And finally we deduce that
$$
\omega (b_0)= \omega (a_0b_1)- \omega (a_1)= \bigl( \omega (a_1)+ \omega (b_2)\bigr) + \omega (b_1)- \omega (a_1)= \omega (b_1b_2) > 3\omega (\varphi _0)/7;
$$
where we used (\ref{degh=3-lemma-label15}) for the above inequality. But this is a contradiction.

\vspace{0.5ex}

$\bullet$ Our next task is to show that the {\sf Case P2} is impossible as well. In order to do that suppose to the contrary and first assume the {\sf Case S1}, when we again have (\ref{degh=3-lemma-label1}). Furthermore consider the subcase
\begin{itemize}
 \item[{\sf Case}]{\sf S1\!\! -3}. $\omega (b_1) < 2\omega (\varphi _0)/7$;
\end{itemize}
cf. (\ref{degh=3-lemma-label2}). In particular we have that $\omega (a_3b_2) < 2\omega (\varphi _0)/7$ and therefore by (E\! :\! 5) it follows that necessarily the inequality of the {\sf Case S1-1} holds as well. Hence it is immediate that both $\omega (a_2b_2)$ and $\omega (a_3b_1)$ are less than $3\omega (\varphi _0)/7$. By (E\! :\! 4) we conclude that necessarily
\begin{equation}
\label{degh=3-lemma-label18}
\omega (a_1)= \omega (b_0).
\end{equation}
And then $\omega (a_1b_0)=2\omega (b_0) >8\omega (\varphi _0)/7$, which by (E\! :\! 1) implies that we again have (\ref{degh=3-lemma-label3B}). Thus in particular
\begin{equation}
\label{degh=3-lemma-label20}
\omega (b_1)=\omega (a_1b_0)- \omega (a_0)> 5\omega (\varphi _0)/7;
\end{equation}
a contradiction. Next consider the subcase
\begin{itemize}
 \item[{\sf Case}]{\sf S1\!\! -4}. $\omega (b_1) > 2\omega (\varphi _0)/7$.
\end{itemize}
Now by (E\! :\! 5) it follows at once that we must have the equality (\ref{degh=3-lemma-label3}). Hence $\omega (a_2b_0)> 6\omega (\varphi _0)/7$. And also the inequality $\omega (a_0b_2)< 4\omega (\varphi _0)/7$ holds. By (E\! :\! 2) it is obvious that necessarily
\begin{equation}
\label{degh=3-lemma-label21}
\omega (a_2b_0)=\omega (a_1b_1),
\end{equation}
and therefore by (\ref{degh=3-lemma-label3}) we again have (\ref{degh=3-lemma-label18}). Further the same argument as before gives that also both (\ref{degh=3-lemma-label3B}) and (\ref{degh=3-lemma-label20}) do hold. Finally we consider (E\! :\! 3). By (\ref{degh=3-lemma-label1}) we have that  $\omega (a_3b_0)< 8\omega (\varphi _0)/7$, by (\ref{degh=3-lemma-label3}) and (\ref{degh=3-lemma-label20}) it is obvious that $\omega (a_2b_1)> 10\omega (\varphi _0)/7$ and by (\ref{degh=3-lemma-label18}) we deduce that also $\omega (a_1b_2)=\omega (b_0b_2)< 8\omega (\varphi _0)/7$. As a conclusion, $\omega (\varphi _3)= \omega (a_2b_1)$; which is impossible.

Next assume the {\sf Case S2}. Note that again (\ref{degh=3-lemma-label8}) holds. Besides we have that $\omega (a_3b_2) > 2\omega (\varphi _0)/7$, and so by (E\! :\! 5) we are in the {\sf Case S1-2} or {\sf Case S1-4}. Here we will again consider the before formulated three sub-subcases: {\sf Case V1} to {\sf Case V3}. Suppose the first one. Then in particular we have the equality (\ref{degh=3-lemma-label12}). Using this we easily get the inequality
\begin{equation}
\label{degh=3-lemma-label22}
\omega (a_0b_2)< \omega (a_2b_0).
\end{equation}
As we have that $\omega (a_2b_0) > 6\omega (\varphi _0)/7$, by (E\! :\! 2) we necessarily have one more time the equality (\ref{degh=3-lemma-label21}). Further note that it is equivalent to the equality
$$
\omega (a_2)-\omega (b_1) = \omega (a_1)-\omega (b_0),
$$
and therefore
\begin{equation}
\label{degh=3-lemma-label23}
\omega (a_1) > \omega (b_0) > 4\omega (\varphi _0)/7.
\end{equation}
Hence, using that $\omega (a_1b_0) > 8\omega (\varphi _0)/7$, by (E\! :\! 1) we deduce that necessarily (\ref{degh=3-lemma-label3B}) holds. And then it is immediate that (\ref{degh=3-lemma-label20}) holds as well. Thus we have that in fact
$$
\omega (a_2) > \omega (b_1) > 5\omega (\varphi _0)/7.
$$
Analogously as for the argument that the {\sf Case S1-4} is impossible, we have the inequalities $\omega (a_2b_1) > 10\omega (\varphi _0)/7 > \omega (a_3b_0)$. Using these by (E\! :\! 3) it follows that $\omega (a_2b_1)= \omega (a_1b_2)$, which is by (\ref{degh=3-lemma-label12}) further equivalent to the equality
\begin{equation}
\label{degh=3-lemma-label24}
\omega (a_1)=\omega (b_1)+ \omega (b_2).
\end{equation}
Hence by (\ref{degh=3-lemma-label20}) it follows that necessarily
$\omega (a_1) > 6\omega (\varphi _0)/7$; cf. (\ref{degh=3-lemma-label23}). Next we conclude that $\omega (a_1b_0) > 10\omega (\varphi _0)/7$, and then the same reason as for (\ref{degh=3-lemma-label20}) gives that moreover $\omega (b_1) > \omega (\varphi _0)$. So by (\ref{degh=3-lemma-label24}) we have that $\omega (a_1) > \omega (\varphi _0)$. At the same time, as we are in the {\sf Case V1}, by (\ref{degh=3-lemma-label8}) it follows that (\ref{degh=3-lemma-label12A}) holds here as well. Furthermore note that $\omega (a_1)< \omega (a_2b_2)$. As a conclusion of all that we have observed it is necessarily
$$
\omega (\varphi _4) =\omega (a_2b_2) > \omega (\varphi _0);
$$
a contradiction.

Next assume the {\sf Case V2}. Note that now we must have the {\sf Case S1-4}. Also note that both (\ref{degh=3-lemma-label6A}) and (\ref{degh=3-lemma-label9}) do hold again. We will further refine our discussion via the following three possibilities:
\begin{itemize}
 \item[{\sf Case}]{\sf W1}. $\omega (a_1) > \omega (b_0)$.
 \item[{\sf Case}]{\sf W2}. $\omega (a_1) < \omega (b_0)$.
 \item[{\sf Case}]{\sf W3}. $\omega (a_1) = \omega (b_0)$.
\end{itemize}

Assume the {\sf Case W1}. Here by (\ref{degh=3-lemma-label6A}) we get (\ref{degh=3-lemma-label10}), and then taking into account (\ref{degh=3-lemma-label9}) it follows at once that the inequality
(\ref{degh=3-lemma-label10A}) holds as well. Hence by (E\! :\! 3) we obtain that necessarily $\omega (a_3b_0)=\omega (a_1b_2)$, which is by (\ref{degh=3-lemma-label8}) further equivalent to the equality $\omega (b_0)=\omega (a_1)$; a contradiction.

Assume the {\sf Case W2}. Using (\ref{degh=3-lemma-label6A}) and (\ref{degh=3-lemma-label9}), by (E\! :\! 4) we have that
\begin{equation}
\label{degh=3-lemma-label25}
\omega (b_0)=\omega (a_3b_1)= 3\omega (b_2).
\end{equation}
Also observe that the inequality $\omega (a_1b_2)< \omega (a_3b_0)$ holds. Hence by (E\! :\! 3) we deduce that necessarily
\begin{equation}
\label{degh=3-lemma-label25A}
\omega (a_3b_0)=\omega (a_2b_1).
\end{equation}
Thus by (\ref{degh=3-lemma-label8}), (\ref{degh=3-lemma-label9}) and (\ref{degh=3-lemma-label25}) we obtain from (E\! :\! 3) that
\begin{equation}
\label{degh=3-lemma-label26}
\omega (a_2)= \omega (b_2b_0) -\omega (b_1)= 2\omega (b_2);
\end{equation}
cf. (\ref{degh=3-lemma-label12}). Next consider (E\! :\! 2). We clearly have that $\omega (a_0)< 4\omega (b_2)$, which is by (\ref{degh=3-lemma-label25}) and (\ref{degh=3-lemma-label26}) equivalent to the inequality (\ref{degh=3-lemma-label22}). Also note that by (\ref{degh=3-lemma-label9}) and (\ref{degh=3-lemma-label26}) we have that
$$
\omega (a_1b_1) < \omega (a_2b_0) = 5\omega (b_2).
$$
As a conclusion it follows that
\begin{equation}
\label{degh=3-lemma-label26B}
\omega (\varphi _2) =\omega (a_2b_0) > 5\omega (\varphi _0)/7;
\end{equation}
a contradiction.

Now assume the {\sf Case W3}. Here it is easy to check that both $\omega (a_2b_0)$ and $\omega (a_0b_2)$ are less than $\omega (a_1b_1)$. Therefore, as we are in the {\sf Case S2}, by (E\! :\! 2) and (\ref{degh=3-lemma-label9}) we obtain that
\begin{equation}
\label{degh=3-lemma-label26C}
\omega (\varphi _2) =\omega (a_1b_1)= \omega (b_0b_1)=\omega (b_0) +2\omega (b_1) > 6\omega (\varphi _0)/7;
\end{equation}
a contradiction.

It remains to treat the {\sf Case V3}. Again we consider the stated three possibilities: {\sf Case W1} to {\sf Case W3}. As we will see below here a new type of argument is needed. First assume the {\sf Case W1}. Then
\begin{equation}
\label{degh=3-lemma-label26A}
\omega (a_1b_0) > 2\omega (b_0) > 8\omega (\varphi _0)/7,
\end{equation}
and so by (E\! :\! 1) both of the equalities (\ref{degh=3-lemma-label3B}) and (\ref{degh=3-lemma-label20}) do hold. Observe that again we have (\ref{degh=3-lemma-label16}), and by
(\ref{degh=3-lemma-label20}) that
\begin{equation}
\label{degh=3-lemma-label28}
\omega (a_1b_1) > \omega (b_0) +5\omega (\varphi _0)/7 > 9\omega (\varphi _0)/7.
\end{equation}
By (E\! :\! 2) necessarily (\ref{degh=3-lemma-label17}) holds. And as a consequence,
\begin{equation}
\label{degh=3-lemma-label29}
\omega (b_2)= \omega (a_1b_1)- \omega (a_0) > 6\omega (\varphi _0)/7.
\end{equation}
Next consider (E\! :\! 5) and then note that by the last inequality we must have that
\begin{equation}
\label{degh=3-lemma-label29}
\omega (a_2)= \omega (b_1)\geq \omega (a_3b_2)= 2\omega (b_2) > 12\omega (\varphi _0)/7.
\end{equation}
Exactly as we did for (\ref{degh=3-lemma-label28}), now we have that $\omega (a_1b_1) > 16\omega (\varphi _0)/7$. And so $\omega (b_2) > 13\omega (\varphi _0)/7$, while $\omega (a_2) > 26\omega (\varphi _0)/7$. Proceeding inductively in particular we get that $\omega (b_2)$ is greater than any natural number, which is impossible.

Assume the {\sf Case W2}. Here first observe that analogously as for (\ref{degh=3-lemma-label29}) we have that $\omega (a_2)= \omega (b_1) > 2\omega (\varphi _0)/7$. Hence it follows that $\omega (a_2b_0) > 6\omega (\varphi _0)/7$. As clearly $\omega (a_2b_0) > \omega (a_1b_1)$, by (E\! :\! 2) we obtain that
\begin{equation}
\label{degh=3-lemma-label30}
\omega (a_2b_0)= \omega (a_0b_2).
\end{equation}
And hence
$$
\omega (b_2)= \omega (a_2b_0)- \omega (a_0) > 3\omega (\varphi _0)/7.
$$
But again as for (\ref{degh=3-lemma-label29}) here we get that $\omega (a_2)=\omega (b_1) > 6\omega (\varphi _0)/7$. Using the last inequality, it follows that $\omega (a_2b_0)> 10\omega (\varphi _0)/7$, and then by (\ref{degh=3-lemma-label30}) it is immediate that $\omega (b_2)> \omega (\varphi _0)$. Proceeding in the same way as for the {\sf Case W1} we conclude that the {\sf Case W2} is impossible as well.

The last one possibility is the {\sf Case W3}. The same argument as for the {\sf Case W1} gives that (\ref{degh=3-lemma-label20}) holds again. Thus, as we are in the {\sf Case V3}, it follows that $\omega (a_2b_1)=2\omega (b_1) > 10\omega (\varphi _0)/7$. At the same time it is obvious that $\omega (a_3b_0)= \omega (a_1b_2)$. Therefore by (E\! :\! 3) one concludes that necessarily $\omega (a_2b_1)\leq \omega (a_3b_0)$. And so it follows that
\begin{equation}
\label{degh=3-lemma-label31}
\omega (a_3)=\omega (b_2) \geq \omega (a_2b_1)- \omega (b_0) \geq \omega (a_2b_1)- \omega (\varphi _0) > 3\omega (\varphi _0)/7.
\end{equation}
Further as for (\ref{degh=3-lemma-label29}) now we have that $\omega (a_2)= \omega (b_1)> 6\omega (\varphi _0)/7$, and so $\omega (a_2b_1)> 12\omega (\varphi _0)/7$. Analogously as for (\ref{degh=3-lemma-label31}) we deduce that necessarily $\omega (b_2)> 5\omega (\varphi _0)/7$. The same conclusion as for the previous two cases gives that the {\sf Case W3} is impossible as well.

\vspace{0.5ex}

$\bullet$ Our last step to finish this proof is to show that the {\sf Case P3} is also impossible. Again suppose to the contrary and first assume the {\sf Case S1}. Observe that now also $3\omega (\varphi _0)/7< \varphi (b_0)< 4\omega (\varphi _0)/7$, and (\ref{degh=3-lemma-label1}) holds. One more time we will treat the subcases {\sf Case S1-3} and {\sf Case S1-4}. We begin with the first one of them. Note that we are in the {\sf Case S1-1} as well and therefore the equality (\ref{degh=3-lemma-label18}) holds, with the same argument as there. Then we have that $\omega (a_1b_0)> 6\omega (\varphi _0)/7$, and so (E\! :\! 1) implies that (\ref{degh=3-lemma-label3B}) holds. Thus we conclude that $\omega (b_1)> 2\omega (\varphi _0)/7$, which is a contradiction; cf. (\ref{degh=3-lemma-label20}).

Next assume the {\sf Case S1-4}. Observe that we again have both (\ref{degh=3-lemma-label3}) and (\ref{degh=3-lemma-label21}), and therefore (\ref{degh=3-lemma-label18}) as well. Hence it is immediate that (\ref{degh=3-lemma-label3B}) holds. By (\ref{degh=3-lemma-label3B}) and (\ref{degh=3-lemma-label18}) it follows that $2\omega (b_0)= \omega (a_0) +\omega (b_1)$. At the same time $\omega (a_0)< 2\omega (b_1)$. So we deduce that
\begin{equation}
\label{degh=3-lemma-label32}
2\omega (b_0) < 3\omega (b_1).
\end{equation}
Now suppose that $\omega (b_2) < \omega (a_3)$. Then we clearly have that $\omega (a_3b_0) > \omega (a_1b_2)$. Also note that by (\ref{degh=3-lemma-label3}) we have the inequality $\omega (a_2b_1) > 4\omega (\varphi _0)/7$. Therefore, as we have that $\omega (a_0b_3) < 4\omega (\varphi _0)/7$, by (E\! :\! 3) it follows that necessarily (\ref{degh=3-lemma-label25A}) holds. Further note that $\omega (a_3b_1) > \omega (a_2b_2)$. Besides by (\ref{degh=3-lemma-label3}), (\ref{degh=3-lemma-label25A}) and (\ref{degh=3-lemma-label32}) we have that
$$
2\omega (b_0)< \omega (a_2)+2 \omega (b_1)= \omega (a_3b_0)+ \omega (b_1);
$$
and so $\omega (b_0)<\omega (a_3b_1)$. Taking into account all that we have obtained, by (E\! :\! 4) it follows that
$$
\omega (\varphi _4)= \omega (a_3b_1) > \omega (b_0) > 3 \omega (\varphi _0)/7;
$$
a contradiction. As the second possibility, suppose that $\omega (b_2) \geq \omega (a_3)$. Then by (\ref{degh=3-lemma-label18}) we have the inequality $\omega (a_3b_0)\leq \omega (a_1b_2)$, and by
(\ref{degh=3-lemma-label3}) the inequality $\omega (a_2b_1)> 4\omega (\varphi _0)/7 > \omega (a_0b_3)$. Also, as we are in the {\sf Case S1}, it follows that
$$
\omega (a_0) + 2\omega (b_2) < 6\omega (\varphi _0)/7 < 3\omega (b_1).
$$
Further, using (\ref{degh=3-lemma-label3}), (\ref{degh=3-lemma-label3B}) and (\ref{degh=3-lemma-label18}), the last inequality is equivalent to the inequality
$$
2\omega (a_1b_2)= 2\omega (b_0b_2)= \omega (a_0b_1) + 2\omega (b_2) < 4\omega (b_1) = 2\omega (a_2b_1);
$$
i.e., we have that (\ref{degh=3-lemma-label5}) holds. Then as a conclusion we have that
\begin{equation}
\label{degh=3-lemma-label32A}
\omega (\varphi _3)= \omega (a_2b_1)= 2\omega (b_1) > 4\omega (\varphi _0)/7;
\end{equation}
a contradiction.

It remains to treat the {\sf Case S2}. Here we have (\ref{degh=3-lemma-label8}). One more time we will consider the three possibilities: {\sf Case V1} to {\sf Case V3}.

First assume the {\sf Case V1}. As we have that $\omega (a_3b_2) > 2\omega (\varphi _0)/7$, by (E\! :\! 5) it is clear that now (\ref{degh=3-lemma-label12}) holds. At this point we further refine our discussion by treating the three possibilities: {\sf Case W1} to {\sf Case W3}. Assume the first one. Then it is immediate that we must have (\ref{degh=3-lemma-label14}). By (\ref{degh=3-lemma-label12}) it follows the inequality $-2\omega (b_2)= -\omega (a_2) < -\omega (b_1)$, and hence by (\ref{degh=3-lemma-label14}) we have that
$$
\omega (b_2)< 3\omega (b_2)- \omega (b_1)= \omega (a_1)- \omega (b_1).
$$
As $\omega (a_0)- \omega (b_0)< \omega (\varphi _0)/7$ it follows at once that $\omega (a_1b_0) > \omega (a_0b_1)$. Using that $\omega (a_1b_0)> 2\omega (b_0)> 6\omega (\varphi _0)/7$, we conclude that
$$
\omega (\varphi _1)= \omega (a_1b_0) > 6\omega (\varphi _0)/7,
$$
which is a contradiction; cf. (\ref{degh=3-lemma-label11A}). Next assume the {\sf Case W2}. Here we have that
\begin{equation}
\label{degh=3-lemma-label33}
\omega (a_2b_0)> \omega (a_1b_1).
\end{equation}
Using the obvious inequality $\omega (a_0)- \omega (b_0)< \omega (b_2)$, by (\ref{degh=3-lemma-label12}) it is immediate that (\ref{degh=3-lemma-label22}) holds as well. As a conclusion,
\begin{equation}
\label{degh=3-lemma-label33A}
\omega (\varphi _2)= \omega (a_2b_0) = 2\omega (b_2) +\omega (b_0) > 5\omega (\varphi _0)/7,
\end{equation}
which is again a contradiction; cf. the argument for (\ref{degh=3-lemma-label26B}). Assume the third possibility, the {\sf Case W3}. As we have the inequality $\omega (a_3b_2)= 2\omega (b_2) > 2\omega (\varphi _0)/7$, by (E\! :\! 5) it is clear that
\begin{equation}
\label{degh=3-lemma-label34}
\omega (a_2)=\omega (a_3b_2)= 2\omega (b_2).
\end{equation}
Next it is obvious that again (\ref{degh=3-lemma-label33}) holds. Also, by (\ref{degh=3-lemma-label34}) the inequality (\ref{degh=3-lemma-label22}) holds. Now the same conclusion as for the {\sf Case W2} works here as well.

As the second step here, assume the {\sf Case V2}. Observe that then we have both (\ref{degh=3-lemma-label6A}) and (\ref{degh=3-lemma-label9}). We proceed analogously as for the {\sf Case V1}, and so let us begin with the {\sf Case W1}. Now it is clear that (\ref{degh=3-lemma-label10}) holds. By considering (E\! :\! 3) first note that
$$
\omega (a_3b_0)=\omega (b_2)+ \omega (b_0)> 4\omega (\varphi _0)/7 > \omega (a_0b_3).
$$
And we clearly have the inequality (\ref{degh=3-lemma-label10X}). Using the equalities (\ref{degh=3-lemma-label9}) and (\ref{degh=3-lemma-label10}) we have at once that the inequality (\ref{degh=3-lemma-label10A}) holds ass well. By all we have noted it follows that
$$
\omega (\varphi _3)= \omega (a_1b_2)= 4\omega (b_2)> 4\omega (\varphi _0)/7;
$$
a contradiction. Next assume the {\sf Case W2}. We have (\ref{degh=3-lemma-label25}). And thus also both
$$
\omega (a_2b_1)< 2\omega (b_1)= 4\omega (b_2)
$$
and
$$
\omega (a_1b_2)< \omega (b_0b_2)= 4\omega (b_2)= \omega (a_3b_0).
$$
Further observe that $\omega (a_0b_3)< 4\omega (\varphi _0)/7 < \omega (a_3b_0)$, and therefore by (E\! :\! 3) we conclude that
$$
\omega (\varphi _3) =\omega (a_3b_0) > 4\omega (\varphi _0)/7,
$$
which is a contradiction; cf. (\ref{degh=3-lemma-label2A}). Assume the third possibility, the {\sf Case W3}. Then we have (\ref{degh=3-lemma-label6}), and the inequality $\omega (a_0b_2)< \omega (a_1b_1)$ as well. Therefore, as we have that $\omega (a_1b_1)= \omega (b_0b_1)> 5\omega (\varphi _0)/7$, by (E\! :\! 2) we deduce that $\omega (\varphi _2)= \omega (a_1b_1)$, which is impossible; cf. (\ref{degh=3-lemma-label26C}).

As the last step assume the {\sf Case V3}. First note that here again (\ref{degh=3-lemma-label15}) holds.
As several times before we treat the three subcases: {\sf Case W1} to {\sf Case W3}. Assume the first one. We clearly have the inequality (\ref{degh=3-lemma-label16}). Moreover, using (\ref{degh=3-lemma-label15}), we have that
$$
\omega (a_0)- \omega (b_0) <\omega (\varphi _0)/7 < \omega (b_2) \leq \omega (b_1) -\omega (b_2),
$$
and therefore
$$
\omega (a_0b_2) < \omega (b_0b_1) < \omega (a_1b_1).
$$
Now by (E\! :\! 2), the last inequality and (\ref{degh=3-lemma-label16}) we conclude that
$$
\omega (\varphi _2) = \omega (a_1b_1) > \omega (b_0) +\omega (b_1) > 5\omega (\varphi _0)/7;
$$
a contradiction. Next assume the {\sf Case W2}. It is obvious that now (\ref{degh=3-lemma-label33}) holds. Also by (\ref{degh=3-lemma-label15}) we have that
$$
\bigl( \omega (a_2) -\omega (b_2) \bigr) +\omega (b_0) \geq \omega (b_2) +\omega (b_0) > 4\omega (\varphi _0)/7 > \omega (a_0) ;
$$
and so the inequality (\ref{degh=3-lemma-label22}) holds. Hence we conclude that necessarily
$\omega (\varphi _2)=\omega (b_0)$, which is impossible; cf. (\ref{degh=3-lemma-label33A}). It remains to treat the {\sf Case W3}. Here it is easy to see that (\ref{degh=3-lemma-label3B}) holds again. We also have that
\begin{equation}
\label{degh=3-lemma-label35}
\omega (a_3b_0) =\omega (a_1b_2) = \omega (b_0b_2) > 4\omega (\varphi _0)/7 > \omega (a_0).
\end{equation}
Furthermore,
\begin{equation}
\label{degh=3-lemma-label36}
\omega (a_2b_1) > \omega (a_3b_0)
\end{equation}
holds as well. Namely by (\ref{degh=3-lemma-label3B}), (\ref{degh=3-lemma-label8}) and (\ref{degh=3-lemma-label15}) the above inequality is equivalent to
$$
2\omega (b_1) > \omega (b_2b_0)= \omega (b_2) +\bigl( \omega (a_0b_1) -\omega (a_1)\bigr) ,
$$
which is further equivalent to the inequality
$$
\omega (b_1) -\omega (b_2) > \omega (a_0) - \omega (a_1)= \omega (a_0) -\omega (b_0).
$$
Here one just have to observe that $\omega (b_1) -\omega (b_2) >\omega (\varphi _0)/7$; and so (\ref{degh=3-lemma-label36}) follows. Finally by (\ref{degh=3-lemma-label35}) and (\ref{degh=3-lemma-label36}) we get (\ref{degh=3-lemma-label32A}), which cannot be. Thus we have completed our proof of the lemma.
\end{proof}

Our Theorem \ref{THMDumas1kind} can in particular be very helpful while studying the reducibility problem for a large class of multivariate polynomials. As one related result we have the following obvious corollary. Note that if our Conjecture is true, then this gives a much more stronger result than for example the \cite[Cor. 4.10]{Ga} does.

\begin{cor}
\label{Dumas1-cor}
Let $A$ be an integral domain and $R=A[X_1,\ldots ,X_k]$ be the ring of $A$-polynomials in $k\geq 1$ variables. Let $\sigma$ be either the total degree $\deg$ on $R$ or $\deg _i$, the degree with respect to some variable $X_i$. Suppose we have $n>1$ and any polynomials $\varphi _0,\varphi _1,\ldots ,\varphi _{n-1}\in R$ so that $\gcd (n,\sigma (\varphi _0))=1$ and also
$$
\sigma (\varphi _j) \leq \Bigl\lfloor \frac{(n-j)}{n}\sigma (\varphi _0)\Bigr\rfloor \qquad \text{for $j=1,\ldots ,n-1$}.
$$
Then for $n\leq 7$ and any $a\in A^{\times}$ the polynomial
$$
f=aX^n+\varphi _{n-1}X^{n-1} +\cdots + \varphi _1X + \varphi  _0
$$
cannot be written as a product $f=gh$ for some polynomials $g,h\in R[X]=A[X,X_1,\ldots ,X_k]$ of degrees at least one.
\end{cor}

With our method sometimes we can get even more than the previous corollary gives. Let us present two simple, but illustrative, examples.

\begin{ex}
\label{Dumas1-ex}
Let $A$ be any integral domain and define the ring of polynomials $R=A[Y,Z]$.

(1) Consider polynomials $\varphi _0,\varphi _1,\psi _0,\psi _1 \in R$ defined by $\varphi _0=Y^4Z^3 +\psi _0$ and $\varphi _1=Y^3Z^3 +\psi _1$, where the total degrees of $\psi _0$ and $\psi _1$ satisfy $\deg \psi _0\leq 4$ and $\deg \psi _1\leq 3$. And then consider the polynomial
$$
f=X^6 +\varphi _1(Y,Z)X +\varphi _0(Y,Z)\in A[X,Y,Z].
$$
We claim that $f$ cannot be written as a product of two nonconstant polynomials. First observe that $\gcd (6,\deg _Y\varphi _0)=2$, $\gcd (6,\deg _Z\varphi _0)=3$ and also that $\deg \varphi _1=6>35/6=5\deg \varphi _0/6$. Thus we clearly cannot apply the previous corollary. But if we choose $\boldsymbol{\nu}=(2,1)$ and consider the corresponding Dumas valuation $\omega =\deg _{\boldsymbol{\nu}} :A[Y,Z]\to \mathbb N_0^{-\infty}$, we have the following:
$$
\omega (\varphi _1)=9 < 5\cdot 11/6 = (6-1)\omega (\varphi _0)/6.
$$
Thus for $f$ we have both of the conditions {\sf (A)} and {\sf (B1)} of our Conjecture fulfilled. And so by Theorem \ref{THMDumas1kind} our claim follows.

(2) Suppose we have polynomials $\varphi _0, \varphi _1, \varphi _2, \varphi _3\in R$ and elements $a_0,a_1,a_2,a_3\in A^{\times}$ satisfying the following. We have $\varphi _0=a_0Y^6Z^5+\psi _0$, where $\psi _0\in R$ is any polynomial such that $\deg \psi _0\leq 6$. We have $\varphi _1=a_1Y^3Z+\psi _1$, where $\psi _1\in R$ satisfies $\deg \psi _1\leq 3$. We have $\varphi _2=a_2Y^2Z+\psi _2$, where $\psi _2\in R$ satisfies $\deg \psi _2\leq 2$. And we have $\varphi _3=a_3YZ^4+\psi _3$, where $\psi _3\in A[Y]$ is such that $\deg \psi _3\leq 10$. Then the polynomial $f\in A[X,Y,Z]$, given by
$$
f=Z^4+\varphi _3Z^3+ \varphi _2Z^2+ \varphi _1Z + \varphi _0,
$$
cannot be written as a product of two nonconstant polynomials. To see this one just have to choose $\boldsymbol{\nu}=(7,1)$ and then define the corresponding Dumas valuation $\omega =\deg _{\boldsymbol{\nu}}$. It remains to apply our theorem. Observe that again we cannot directly apply the above corollary.
\end{ex}

In fact it is worth to formulate the following generalization of the above given corollary.

\begin{cor}
\label{Dumas1-cor2}
Let $A$ be an integral domain and $R=A[X_1,\ldots ,X_k]$, for $k\geq 2$. Suppose we have $n>1$, some $\boldsymbol{\nu}\in \mathbb N_0^k$ and any polynomials $\varphi _0,\varphi _1,\ldots ,\varphi _n\in R$ so that $\deg _{\boldsymbol{\nu}} \varphi _n=0$, $\mu =\deg _{\boldsymbol{\nu}}\varphi _0$ is a positive number satisfying $\gcd (n,\mu )=1$ and also
$$
\deg _{\boldsymbol{\nu}} \varphi _j \leq \Bigl\lfloor \frac{(n-j)}{n}\mu \Bigr\rfloor \qquad \text{for $j=1,\ldots ,n-1$}.
$$
Then for $n\leq 7$ the polynomial $f=\sum _{j=0}^n\varphi _jX^j$ is irreducible in the ring of polynomials $A[X,X_1,\ldots ,X_k]$.
\end{cor}

We conclude this section with few more remarks.

\begin{rmk}
\label{Dumas1-rmk}
Let $A$ be any integral domain and define the ring of polynomials $R=A[X,Z]$. Take $m,n,i\in\mathbb N$ where $i<m$, and any polynomials $\varphi _0,\varphi _i\in R$ satisfying $\deg _X\varphi _0=n$ and $\deg _X\varphi _i \leq (m-i)n/m$. If we consider the Dumas valuation $\omega =\deg _X$ on $R$, then for $m\leq 7$ our Theorem \ref{THMDumas1kind} implies at once that the polynomial
\begin{equation}
\label{Dumas1-rmk-lab1}
f=Y^m +\varphi _i(X,Z)Y^i +\varphi _0(X,Z)
\end{equation}
cannot be written as a product of two nonconstant polynomials provided that $\gcd (m,n)=1$.

Note that Gao on p.~517 of \cite{Ga} consider a special case of our polynomial $f$ by taking $\varphi _i(X,Z)=Z^j$ and $\varphi _0(X,Z)=X^n +X^uZ^v +Z^w+1$, where $u<n$ while $j$, $v$ and $w$ are arbitrary. (Of course with no restrictions on $n$, as we have.) And then concludes via polytope method that this special $f$ is (absolutely) irreducible if $A$ is a field. Also note that here for the chosen polynomials we have $\deg _X\varphi _i=0$ and $\deg _X\varphi _0=n$; and therefore the condition {\sf (B1)} is trivially fulfilled.
\end{rmk}

Perhaps this might be a good place for one more observation. Assume the setting as in the above Remark. Then take a positive divisor $d$ of $n$ so that $\gcd (m,n/d)=1$. And then consider the $\Gamma$-Dumas valuation $\omega _1= \frac{1}{d}\deg _X :R\to \Gamma$, for $\Gamma =\frac{1}{d}\mathbb N_0$; see Lemma \ref{ba0}. It would be interesting to know whether our Conjecture is true if a Dumas valuation there is replaced by a $\Gamma$-Dumas valuation, for $\Gamma$ as above. Namely, in case of the positive answer one could at least sometimes be able to drop the ``$\gcd$-assumption'' in the condition {\sf (A)}. Polynomials of the type as in (\ref{Dumas1-rmk-lab1}) might be appropriate for testing this.

\section{The case of Dumas valuations of the second kind}
\label{Dumas2}
In this section we complete a proof od Theorem \ref{mainTHM}, by treating the case when $\omega$ is a Dumas valuation of the second kind.  Our goal is to prove the following.

\begin{thm}
\label{THMDumas2kind}
Let $R$ be an integral domain, $n>1$ a natural number and elements $\varphi _0,\varphi _1,\ldots ,\varphi _n\in R$. Suppose $\omega$ is a Dumas valuation of the second kind, defined on $R$ and satisfying the conditions {\sf (A)} and {\sf (B2)} of the Conjecture. Then for every $n\leq 6$ the polynomial $f=\sum _{j=0}^n\varphi _jX^j$ cannot be written as a product of two nonconstant polynomials from $R[X]$.
\end{thm}

Our main task here is to prove the following analogue of Lemma \ref{degh=2-lemma}. As it will be seen, our arguments for the lemma below are similar to the ones of the mentioned lemma. And so many details will be omitted. One more time Lemma \ref{ba1}, and in particular its claim (iv), will play a crucial role.

\begin{lem}
\label{2Ddegh=2-lemma}
Let $\omega$ be a Dumas valuation of the second kind, defined on an integral domain $R$. Let $f\in R[X]$ be a polynomial as in the Conjecture, of degree $n\geq 4$. Then there are no polynomials $g,h\in R[X]$ such that $f=gh$ and $\deg h=2$.
\end{lem}

\begin{proof}
Again suppose to the contrary, and let the polynomials $g$ and $h$ satisfying $f=gh$ be as in the proof of Lemma \ref{degh=2-lemma}. In what follows we refer all the time to the coefficients equalities (E\! :\! 1) to (E\! :\! \text{n\!\! --\! 1}) given there. And we will treat the same three possibilities formulated as there: {\sf Case P1} to {\sf Case P3}.

$\bullet$ Let us begin by showing that the {\sf Case P1} is impossible. And so suppose to the contrary. We distinguish the same two subcases, the {\sf Case S1} and {\sf Case S2}, as before.

Suppose the {\sf Case S1}; i.e., that $\omega (b_1) <\omega (\varphi _0)/n$. Then it is immediate that (\ref{degh=2-lemma-label1}) holds here as well. Further we refine our analysis by considering the following three sub-subcases:

\begin{itemize}
 \item[{\sf Case}]{\sf S1\!\! -1}. $\omega (b_0) < \omega (a_{n-3}b_1)$.
 \item[{\sf Case}]{\sf S1\!\! -2}. $\omega (b_0) > \omega (a_{n-3}b_1)$.
 \item[{\sf Case}]{\sf S1\!\! -3}. $\omega (b_0) = \omega (a_{n-3}b_1)$.
\end{itemize}

\vspace{0.5ex}

\noindent
Suppose the first sub-subcase. Then by the condition {\sf (B2)} and (E\! :\! \text{n\!\! --\! 2}) it is clear that
\begin{equation}
\label{2Ddegh=2-lemma-label1}
\omega (a_{n-4}) = \omega (b_0);
\end{equation}
cf. the {\sf Case S2-3} in the proof of Lemma \ref{degh=2-lemma}. Next, as we have that $\omega (a_{n-3}b_1)= 2\omega (b_1)$, it is clear that
$$
\omega (a_{n-4}b_1)= \omega (a_{n-3}b_0)< 3\omega (\varphi _0)/n.
$$
Therefore by (E\! :\! \text{n\!\! --\! 3}) we obtain that
\begin{equation}
\label{2Ddegh=2-lemma-label2}
\omega (a_{n-5}) \geq \omega (a_{n-3}b_0)= \omega (b_0) +\omega (b_1).
\end{equation}
Then consider (E\! :\! \text{n\!\! --\! 4}). By (\ref{2Ddegh=2-lemma-label2}) and then (\ref{2Ddegh=2-lemma-label1}) it follows that $\omega (a_{n-5}b_1)> \omega (a_{n-4}b_0)$. As we have that $\omega (a_{n-4}b_0)< 4\omega (b_1) < 4\omega (\varphi _0)/n$, it is immediate that necessarily (\ref{degh=2-lemma-label17}) holds. Further consider (E\! :\! \text{n\!\! --\! 5}). Here by (\ref{degh=2-lemma-label1}), (\ref{degh=2-lemma-label17}) and (\ref{2Ddegh=2-lemma-label2}) one can easily see that $\omega (a_{n-6}b_1)\leq  \omega (a_{n-5}b_0)$. Therefore,
$$
\omega (a_{n-7})\geq \omega (a_{n-6}b_1)= 2\omega (b_0)+ \omega (b_1).
$$
Proceeding by induction we prove that for $k\geq 2$ the following hold:
\begin{equation}
\label{2Ddegh=2-lemma-label3}
\omega (a_{n-2k})= (k-1) \omega (b_0) \quad \text{ and } \quad \omega (a_{n-2k-1})\geq (k-1) \omega (b_0) +\omega (b_1);
\end{equation}
cf. the Claim 1 in the proof of Lemma \ref{degh=2-lemma}. Now suppose $n=2m$ is even. Then
$$
\omega (a_0) =(m-1)\omega (b_0) < (2m-2) \omega (b_1) <(n-2)\omega (\varphi _0)/n;
$$
a contradiction to the {\sf Case P1}. Next suppose $n=2m-1$ is odd. Then $\omega (a_1) =(m-2)\omega (b_0)$ and therefore $\omega (a_1b_0) <(n-1)\omega (\varphi _0)/n$. By the inequality in (\ref{2Ddegh=2-lemma-label3}) we also have that $\omega (a_0) \geq (m-2)\omega (b_0)+ \omega (b_1)$ and so
$$
\omega (a_0b_1) > (m-1)\omega (b_0) =\omega (a_1b_0).
$$
Finally, by (E\! :\! 1) it follows that $\omega (\varphi _1) =\omega (a_1b_0)$, which is impossible.

Now consider the {\sf Case S1-2}. Note that by the inequality $\omega (a_{n-3}b_1)< \omega (b_0)< 2\omega (\varphi _0)/n$, (E\! :\! \text{n\!\! --\! 2}) gives that (\ref{degh=2-lemma-label2}) holds. A similar argument as for the {\sf Case P1} and the subcase {\sf Case S2} in the proof of Lemma  \ref{degh=2-lemma} gives that (\ref{degh=2-lemma-label4}) holds as well. Then in particular we have that
$$
\omega (a_0)= (n-2)\omega (b_1)< (n-2)\omega (\varphi _0)/n;
$$
a contradiction.

It remains to treat the third possibility, the {\sf Case S1-3}. Here it is immediate that
$\omega (a_{n-4}) \geq 2\omega (b_1)$. Hence it follows that
$$
\omega (a_{n-4}b_1) \geq 3\omega (b_1) =\omega (a_{n-3}b_0),
$$
where at the same time $\omega (a_{n-3}b_0)< 3\omega (\varphi _0)/n$. Thus by (E\! :\! \text{n\!\! --\! 3}) we conclude that
$$
\omega (a_{n-5}) \geq \omega (a_{n-3}b_0) =3\omega (b_1).
$$
By induction one can show that
\begin{equation}
\label{2Ddegh=2-lemma-label4}
\omega (a_{n-k})\geq (k-2) \omega (b_1), \quad \text{for $k=4,\ldots ,n$};
\end{equation}
cf. (\ref{degh=2-lemma-label20}). A similar argument as for the Claim 2 in the proof of Lemma \ref{degh=2-lemma} gives that the following hold.

\begin{claim}
We have \textup{(\ref{2Ddegh=2-lemma-label4})} and there is no $4\leq k<n$ such that for both $k$ and $k+1$ we have the strict inequalities in \textup{(\ref{2Ddegh=2-lemma-label4})}.
\end{claim}
Thus in particular we have that $\omega (a_0)=(n-2)\omega (b_1)$ or $\omega (a_1)= (n-3)\omega (b_1)$. The former possibility clearly cannot hold. And for the latter one  we have that $\omega (a_1b_0)=(n-1)\omega (b_1)< (n-1)\omega (\varphi _0)/n$. So by (E\! :\! 1) it is obvious that $\omega (a_1b_0)= \omega (a_0b_1)$ and then we deduce that again the first possibility necessarily holds; which we have shown that is impossible.

Next suppose the {\sf Case S2}; i.e., that $\omega (b_1)> \omega (\varphi _0)/n$.
Then $\omega (a_0b_1)> (n-1)$ $\omega (\varphi _0)/n$ and therefore by (E\! :\! 1) it follows that necessarily $\omega (a_1b_0)\geq (n-1)\omega (\varphi _0)/n$. Hence we deduce that
$$
\omega (a_1) \geq (n-1)\omega (\varphi _0)/n -\omega (b_0) > (n-3)\omega (\varphi _0)/n.
$$
Proceeding by induction we show that
\begin{equation}
\label{2Ddegh=2-lemma-label5}
\omega (a_j) > \frac{n-j-2}{n} \omega (\varphi _0), \quad \text{for $1\leq j\leq n-3$};
\end{equation}
cf. (\ref{degh=2-lemma-label0}). Thus in particular we have that both $\omega (a_{n-4}) > 2\omega (\varphi _0)/n$ and $\omega (a_{n-3}) > \omega (\varphi _0)/n$. Hence it follows that both $\omega (a_{n-3}b_1)$ and $\omega (a_{n-4}b_2)$ are greater than $2\omega (\varphi _0)/n$. By (E\! :\! \text{n\!\! --\! 2}) we conclude that necessarily
$$
\omega (b_0) =\omega (a_{n-2}b_0)\geq 2\omega (\varphi _0)/n,
$$
which cannot be.

\vspace{0.5ex}

$\bullet$ Next we prove that the {\sf Case P2} is impossible. For that purpose again suppose to the contrary and note that now $\omega (b_0)> (n-2)\omega (\varphi _0)/n$. We will consider the same two subcases, the {\sf Case P2-1} and {\sf Case P2-2}, as in the proof of Lemma \ref{degh=2-lemma}.

Assume the first one subcase; i.e., that $\omega (a_1)> \omega (\varphi _0)/n$. Then $\omega (a_1b_0)> (n-1)\omega (\varphi _0)/n$ and so by (E\! :\! 1) it is necessarily $\omega (a_0b_1)\geq  (n-1)\omega (\varphi _0)/n$. Hence, using the fact $n\geq 4$, it is immediate that
\begin{equation}
\label{2Ddegh=2-lemma-label6}
\omega (b_1)> (n-3)\omega (\varphi _0)/n \geq \omega (\varphi _0)/n;
\end{equation}
i.e., we are in the {\sf Case S2}. Now by (E\! :\! \text{n\!\! --\! 1}) we have that
\begin{equation}
\label{2Ddegh=2-lemma-label7}
\omega (a_{n-3})\geq \omega (\varphi _0)/n.
\end{equation}
Further observe that $\omega (a_{n-2}b_0)> 2\omega (\varphi _0)/n$, and also by (\ref{2Ddegh=2-lemma-label6}) and (\ref{2Ddegh=2-lemma-label7}) that $\omega (a_{n-3}b_1)> 2\omega (\varphi _0)/n$. Thus by (E\! :\! \text{n\!\! --\! 2}) we conclude that
$$
\omega (a_{n-4})= \omega (a_{n-4}b_2)\geq 2\omega (\varphi _0)/n.
$$
Proceeding by induction we obtain that
$$
\omega (a_j) \geq \frac{n-j-2}{n} \omega (\varphi _0), \quad \text{for $0\leq j\leq n-3$};
$$
cf. (\ref{2Ddegh=2-lemma-label5}). In particular, using the above inequality for $j=0$, we have that $$
\omega (a_0) \geq (n-2)\omega (\varphi _0)/n \geq 2\omega (\varphi _0)/n >\omega (a_0);
$$
a contradiction.

As the second step, assume the {\sf Case P2-2}. Here we have that
\begin{equation*}
\omega (a_3b_0) \geq \omega (b_0)> (n-2)\omega (\varphi _0)/n > \omega (\varphi _0)/n > \omega (a_1b_2).
\end{equation*}
As it is obviously
$$
\omega (\varphi _3)\geq (n-3)\omega (\varphi _0)/n \geq \omega (\varphi _0)/n,
$$
by (E\! :\! 3) it follows that necessarily $\omega (a_1)=\omega (a_2b_1)$. And therefore
$$
\omega (a_2)=\omega (a_1)-\omega (b_1)\leq \omega (a_1)< \omega (\varphi _0)/n.
$$
Similarly by considering (E\! :\! 4) we can show that then $\omega (a_2)=\omega (a_3b_1)$ and hence that $\omega (a_3)< \omega (\varphi _0)/n$. In fact by induction one can prove that
$$
\omega (a_j) <\frac{\omega (\varphi _0)}{n}, \quad \text{for $1\leq j\leq n-3$}.
$$
Thus in particular we deduce that $\omega (a_{n-3}b_2) < \omega (\varphi _0)/n$ and so by
(E\! :\! \text{n\!\! --\! 1}) that necessarily
\begin{equation*}
\omega (a_{n-3})=\omega (b_1) < \omega (\varphi _0)/n.
\end{equation*}
Now by the last inequality we have that
$$
\omega (a_0b_1)< 3\omega (\varphi _0)/n \leq (n-1)\omega (\varphi _0)/n,
$$
and so by (E\! :\! 1) we deduce that $\omega (a_1b_0)= \omega (a_0b_1)$; i.e., that
\begin{equation}
\label{2Ddegh=2-lemma-label8}
\omega (b_0)- \omega (a_0)= \omega (b_1)- \omega (a_1).
\end{equation}
But as we are in the {\sf Case P2-2} it is clear that $\omega (b_1)- \omega (a_1)< \omega (\varphi _0)/n$. At the same time for $n\geq 5$ we obtain that
$$
\omega (b_0)- \omega (a_0) > (n-2)\omega (\varphi _0)/n -2\omega (\varphi _0)/n \geq \omega (\varphi _0)/n.
$$
But thus we get that (\ref{2Ddegh=2-lemma-label8}) is impossible. Observe that the case $n=4$ must be treated separately. This easy checking is left to the reader.

\vspace{0.5ex}

$\bullet$ Our last step is to prove that the {\sf Case P3} is impossible as well. Suppose to the contrary and consider the {\sf Case S1}. First suppose that
\begin{itemize}
\item[]  $\omega (a_{n-4}) <\omega (b_0)$;
\end{itemize}
which is in fact the {\sf Case S2-1} in the proof of Lemma \ref{degh=2-lemma}. As (\ref{degh=2-lemma-label1}) holds, then we have that $\omega (a_{n-3}b_1)< 2\omega (\varphi _0)/n$ and therefore it immediately follows that (\ref{degh=2-lemma-label2}) holds. Furthermore with a little effort one can check that (\ref{degh=2-lemma-label4}) holds as well. Now consider (E\! :\! 1). As we have that $\omega (a_0b_1)= (n-1)\omega (b_1)< (n-1)\omega (\varphi _0)/n$, one moretime it follows the equality $\omega (a_1b_0)= \omega (a_0b_1)$. Hence we deduce that $\omega (b_0) =2\omega (b_1) < 2\omega (\varphi _0)/n$, which is impossible. Next suppose
\begin{itemize}
\item[]  $\omega (a_{n-4}) \geq\omega (b_0)$;
\end{itemize}
which covers both the {\sf Case S2-2} and {\sf Case S2-3} in the proof of Lemma \ref{degh=2-lemma}. Here we have that
$$
\omega (a_{n-4}b_2)\geq \omega (a_{n-2}b_0)> 2\omega (\varphi _0)/n > \omega (a_{n-3}b_1),
$$
and therefore by (E\! :\! \text{n\!\! --\! 2}) we obtain that $\omega (\varphi _2)= \omega (a_{n-3}b_1)$, which is impossible by the last inequality.

What remains to do is to treat the {\sf Case S2}. Now by (E\! :\! \text{n\!\! --\! 1}) it is clear that necessarily $\omega (a_{n-3})\geq \omega (\varphi _0)/n$. By inductive argument one can show that
$$
\omega (a_{n-k}) \geq \frac{k-2}{n} \omega (\varphi _0), \quad \text{for $3\leq k\leq n$};
$$
cf. (\ref{2Ddegh=2-lemma-label5}). In particular for $k=n$ we get the inequality $\omega (a_0) \geq (n-2)\omega (\varphi _0)/n$; a contradiction to the {\sf Case P3}.
Thus we have our lemma proved.
\end{proof}

It remains to establish the following fact.

\begin{lem}
\label{2Ddegh=3-lemma}
Let $\omega$ be a Dumas valuation of the second kind, defined on an integral domain $R$. Let $f\in R[X]$ be a polynomial as in the Conjecture, of degree $n=6$. Then there are no polynomials $g,h\in R[X]$ such that $f=gh$ and $\deg h=3$.
\end{lem}

\begin{proof}
We will again use the same notation as before. Now suppose that such $g=a_3X^3+a_2X^2+a_1X+a_0$ and $h=b_3X^3+b_2X^2+b_1X+b_0$ do exist; i.e., analogously as in the proof of Lemma \ref{degh=3-lemma}, we have the following equalities:
 \begin{equation*}
  \begin{split}
    a_1b_0+a_0b_1 & =\varphi _1 \qquad\qquad\qquad\qquad\qquad\qquad   (E\! :\! 1) \\
    a_2b_0 + a_1b_1 +a_0b_2 & =\varphi _2 \qquad\qquad\qquad\qquad\qquad\qquad   (E\! :\! 2) \\
    a_3b_0 +a_2b_1 +a_1b_2 + a_0b_3 & =\varphi _3 \qquad\qquad\qquad\qquad\qquad\qquad  (E\! :\! 3) \\
    a_3b_1 +a_2b_2 + a_1b_3 & =\varphi _4 \qquad\qquad\qquad\qquad\qquad\qquad  (E\! :\! 4) \\
    a_3b_2 +a_2b_3 & =\varphi _5 \qquad\qquad\qquad\qquad\qquad\qquad  (E\! :\! 5)
   \end{split}
  \end{equation*}
Consider the two possibilities, that are analogues of the previously formulated ones:
\begin{itemize}
 \item[{\sf Case}]{\sf P1}. $\omega (a_0) > \omega (\varphi _0)/2$.
 \item[{\sf Case}]{\sf P2}. $\omega (a_0) < \omega (\varphi _0)/2$.
\end{itemize}

\vspace{0.5ex}

We will treat the second one, where we note that thus it will be settled the first one case as well. (Namely, as we have $s=3=t=6/2$ the roles of the polynomials $g$ and $h$ are ``symmetric''; i.e., we just have to ``replace the coefficient symbols $a_i$ and $b_i$'' in our arguments given below.) Then assume the following subcase:
\begin{itemize}
 \item[{\sf Case}]{\sf S1}. $\omega (b_2) < \omega (\varphi _0)/6$.
\end{itemize}

\vspace{0.5ex}

At the beginning note that by (E\! :\! 5) we have the equality
\begin{equation}
\label{2Ddegh=3-lemma-label1}
\omega (a_2) = \omega (b_2).
\end{equation}
Further, consider this sub-subcase (cf. {\sf Case S1-3} in the proof of Lemma \ref{degh=3-lemma}):
\begin{itemize}
 \item[{\sf Case}] $\omega (b_1) < \omega (\varphi _0)/3$.
\end{itemize}
As now we have the inequality $\omega (a_0b_1) < 5\omega (\varphi _0)/6$, by (E\! :\! 1) it is immediate that $\omega (a_1b_0)=\omega (a_0b_1)$. Hence it is clear that in particular
\begin{equation}
\label{2Ddegh=3-lemma-label2}
\omega (b_1) > \omega (a_1).
\end{equation}
Next, by (\ref{2Ddegh=3-lemma-label1}) it follows that $\omega (a_2b_2) < \omega (\varphi _0)/3$. And then by (E\! :\! 4) and the inequality $\omega (a_3b_1) > \omega (a_1b_3)$, which is plain by (\ref{2Ddegh=3-lemma-label2}), we deduce that
\begin{equation}
\label{2Ddegh=3-lemma-label3}
\omega (a_1) =\omega (a_2b_2)= 2\omega (b_2).
\end{equation}
Further we have that $\omega (a_0b_3)< \omega (a_3b_0)$ and also by (\ref{2Ddegh=3-lemma-label1}) and (\ref{2Ddegh=3-lemma-label2}) the inequality $\omega (a_1b_2)< \omega (a_2b_1)$. Therefore by (E\! :\! 3) it is immediate that
\begin{equation}
\label{2Ddegh=3-lemma-label4}
\omega (a_0) =\omega (a_1b_2)= 3\omega (b_2).
\end{equation}
Now consider (E\! :\! 2). Observe that $\omega (a_0b_2)< 2\omega (\varphi _0)/3$ and $\omega (a_2b_0)> \omega (a_0b_2)$. And also by (\ref{2Ddegh=3-lemma-label2}), (\ref{2Ddegh=3-lemma-label3}) and (\ref{2Ddegh=3-lemma-label4}) we have that
$$
\omega (a_1b_1) > 2\omega (a_1) = 4\omega (b_2)= \omega (a_0b_2).
$$
Thus we obtain that $\omega (\varphi _2)=\omega (a_0b_2)< 2\omega (\varphi _0)/3$, which is impossible.

As the second possibility we treat the following analogue of the {\sf Case S1-4}, established in the proof of Lemma \ref{degh=3-lemma}):
\begin{itemize}
 \item[{\sf Case}] $\omega (b_1) > \omega (\varphi _0)/3$.
\end{itemize}
Using that $\omega (a_2b_2) < \omega (\varphi _0)/3 < \omega (a_3b_1)$, by (E\! :\! 4) we again have (\ref{2Ddegh=3-lemma-label3}). And therefore it follows at once that (\ref{2Ddegh=3-lemma-label2}) holds as well. We proceed in exactly the same way as for the previous sub-subcase.

Next assume this subcase:
\begin{itemize}
 \item[{\sf Case}]{\sf S2}. $\omega (b_2) > \omega (\varphi _0)/6$.
\end{itemize}

\vspace{0.5ex}

Now by (E\! :\! 5) we clearly have that necessarily
\begin{equation}
\label{2Ddegh=3-lemma-label5}
\omega (a_2) \geq \omega (\varphi _0)/6.
\end{equation}
Again we consider the same two sub-subcases as for the {\sf Case S1}. Let us begin with the first one; i.e., when $\omega (b_1)< \omega (\varphi _0)/3$. Then the same argument as for (\ref{2Ddegh=3-lemma-label2}) gives that this inequality holds here as well. As one more time we have that $\omega (a_3b_1)> \omega (a_1b_3)$ and by (\ref{2Ddegh=3-lemma-label5}) the inequality
$\omega (a_2b_2)> \omega (\varphi _0)/3$, by (E\! :\! 4) it follows that necessarily
\begin{equation}
\label{2Ddegh=3-lemma-label6}
\omega (a_1)\geq \omega (\varphi _0)/3.
\end{equation}
And therefore by (\ref{2Ddegh=3-lemma-label2}) we have that also $\omega (b_1)> \omega (\varphi _0)/3$; a contradiction.

It remains to treat the second sub-subcase, when $\omega (b_1)> \omega (\varphi _0)/3$. Then we again have (\ref{2Ddegh=3-lemma-label6}). Now consider (E\! :\! 3). Observe that by using (\ref{2Ddegh=3-lemma-label5}) and (\ref{2Ddegh=3-lemma-label6}) we obtain at once that $\omega (a_1b_2)$, $\omega (a_2b_1)$ and $\omega (a_3b_0)$ are greater than $\omega (\varphi _0)/2$. And so it follows that
$$
\omega (\varphi _3) =\omega (a_0b_3) =\omega (a_0)< \omega (\varphi _0)/2,
$$
which is impossible. This concludes our proof of the lemma.
\end{proof}

As we did by Corollary \ref{Dumas1-cor2} for the Dumas valuations of the first kind given by $\omega =\deg _{\boldsymbol{\nu}}$ here we state its close analogue for the $\codeg _{\boldsymbol{\nu}}$. Also recall that now our method gives something novel only for polynomials in three or more variables.

\begin{cor}
\label{Dumas2-cor1}
Let $A$ be an integral domain and $R=A[X_1,\ldots ,X_k]$, for $k\geq 2$. Suppose we have $n>1$, some $\boldsymbol{\nu}\in \mathbb N_0^k$ and any polynomials $\varphi _0,\varphi _1,\ldots ,\varphi _n\in R$ so that $\codeg _{\boldsymbol{\nu}} \varphi _n=0$, $\mu =\codeg _{\boldsymbol{\nu}} \varphi _0$ is a positive number satisfying $\gcd (n,\mu )=1$ and also
\begin{equation}
\label{Dumas2-cor1-label1}
\codeg _{\boldsymbol{\nu}} \varphi _j \geq \Bigl\lceil \frac{(n-j)}{n}\mu \Bigr\rceil \qquad \text{for $j=1,\ldots ,n-1$}.
\end{equation}
Then for $n\leq 6$ the polynomial $f=\sum _{j=0}^n\varphi _jX^j$ is irreducible in the ring of polynomials $A[X,X_1,\ldots ,X_k]$.
\end{cor}

Let us conclude this section by one more instructive example.

\begin{ex}
\label{Dumas2-ex1}
Let $A$ be any integral domain and define $R=A[Y,Z]$. Let $\varphi _3\in R$ be a polynomial such that its constant term is nonzero. Let also $\psi _0,\psi _1,\psi _2\in R$ be polynomials satisfying the following. The total degree of each monomial of $\psi _0$ is greater than $4$; i.e., $\codeg \psi _0\geq 4$. And also assume that both $\codeg \psi _1\geq 3$ and $\codeg \psi _2\geq 2$. We claim that for any $a,b,c\in A^{\times}$ the polynomial
$$
f=\varphi _3(Y,Z)X^3+\bigl( aY^2+\psi _2(Y,Z)\bigr) X^2 + \bigl( bZ +\psi _1(Y,Z)\bigr) X + \bigl( cYZ +\psi _0(Y,Z)\bigr)
$$
is irreducible in the ring $A[X,Y,Z]$.

First observe that here we cannot choose the Dumas valuation $\omega =\codeg$, as then for $\varphi _1=bZ +\psi _1$ we have $\omega (\varphi _1)=1$ while for $\varphi _0=cYZ+\psi _0$ we have that $\omega (\varphi _0)=2$. Therefore (\ref{Dumas2-cor1-label1}) for $j=1$ does not hold. We also cannot take neither $\omega =\codeg _Y$ nor $\omega =\codeg _Z$. For example in the former possibility we can have the monomial $eZ^4$, for some $e\in A^{\times}$, and therefore $\omega (\varphi _0)=0$; which cannot be. But we can take $\boldsymbol{\nu}=(1,3)$ and then consider the Dumas valuation $\omega =\codeg _{\boldsymbol{\nu}}$. Then we have $\omega (\varphi _0)=4$, $\omega (\varphi _1)=3$ and for the polynomial $\varphi _2=aY^2+\psi _2$ the equality
$\omega (\varphi _2)=2$. Thus we also have both
$$
\omega (\varphi _1) \geq \Bigl\lceil \frac{(3-1)}{3}\omega (\varphi _0) \Bigr\rceil =3 \quad \text{ and } \quad \omega (\varphi _2) \geq \Bigl\lceil \frac{(3-2)}{3}\omega (\varphi _0) \Bigr\rceil =2.
$$
So we can apply the above corollary.
\end{ex}

\section{Polynomials over orders in algebraic number fields}
\label{Orders}

As we already noted in the Introduction, the classical Eisenstein-Dumas irreducibility criterion holds true if $\mathbb Z$ is replaced by any UFD. Now we would like to formulate its partial generalization which works within one very interesting and quite large class of integral domains, that often will not be UFDs. These are orders in algebraic number fields; see \cite[Ch$.$ I, \S 12]{N}.

Here first recall the following. Given a finite field extension $\mathbb K|\mathbb F$, for any $x\in \mathbb K$ consider an $\mathbb F$-linear map $f_x:\mathbb K\to \mathbb K$, $f_x(y)=xy$. Then the norm map $N=N_{\mathbb K|\mathbb F} :\mathbb K\to \mathbb F$ is defined by $N(x)=\det f_x$.

Assume $\mathbb K$ is an algebraic number field and let $\mathcal O_{\mathbb K}$ be its ring of integers. Let us emphasize that in what follows we always assume that our algebraic number field is contained in $\mathbb C$. Suppose $R\leq \mathcal O_{\mathbb K}$ is an order. In particular if $R\neq \mathcal O_{\mathbb K}$, then it is well known that $R$ is not a UFD; see, e.g., \cite[Ex$.$ 6 for Ch$.$ II]{FT}. Thus for the set of all irreducibles $\Irr R$ and the set of all primes $\Pr R$, for such $R$, we have the strict inclusion $\Pr R\subset \Irr R$; see \cite[Prop$.$ 1.6]{Sch}. Also recall some standard facts about the maximal orders $R=\mathcal O_{\mathbb K}$; see \cite[Ch$.$ VII, \S\S 2 and 3]{Bo}. These rings are Dedekind domains. A Dedekind domain is a UFD if and only if it is a principal ideal domain (PID). Thus we in particular have that the following are equivalent: (a) The class number $h_{\mathbb K}$ equals $1$; (b) $\mathcal O_{\mathbb K}$ is a UFD; (c) $\mathcal O_{\mathbb K}$ is a PID. In any case, whether $R$ is the maximal order or not, it is desirable to gain our knowledge about the corresponding sets of primes $\Pr R$. For that purpose let us state  here the following lemma, which is one of the main results of \cite{Sch}; see Remark \ref{Orders-rmk1} for related observations.

\begin{lem}
\label{Orders-lem1}
Let $\mathbb K$ be an algebraic number field and assume $R$ is an order in $\mathbb K$. If $\mathfrak p \in R$ is an element such that for the norm $N=N_{\mathbb K|\mathbb Q}$ we have $N(\mathfrak p )=p$, for some $p\in \Pr \mathbb Z$, then $\mathfrak p\in \Pr R$.
\end{lem}

\noindent
Observe that this lemma is a nontrivial generalization of the known fact which claims that for $R=\mathcal O_{\mathbb K}$ and $\mathfrak p,p$ as above we can conclude that $\mathfrak p \in \Irr R$; see \cite[Thm. 9.24]{NZM} and \cite[Thm. 9.2.3]{AW}.

For later use recall some further standard facts. Suppose $\mathbb K|\mathbb Q$ is an algebraic number field of degree $m$. Then we have $m$ embeddings $\sigma _j:\mathbb K\to \mathbb C$, for $j=0,1,\ldots ,m-1$, over the inclusion $\mathbb Q \hookrightarrow \mathbb C$; where we can put $\sigma _0$ to be the identity. Here the norm $N=N_{\mathbb K|\mathbb Q}: \mathbb K\to \mathbb Q$ is given as $N(x)=\prod _{i=0}^{m-1}\sigma _i(x)$, for $x\in\mathbb K$. If $\mathbb K$ is moreover a normal extension, then by $G_{\mathbb K} ={\it Gal}(\mathbb K|\mathbb Q)$ denote its Galois group; i.e., $G_{\mathbb K}=\{ \sigma _0,\sigma _1,\ldots ,\sigma _{m-1}\}$. Also observe that $N(x)=N(\sigma (x))$ for every $\sigma \in G_{\mathbb K}$ and recall that $N(x)\in \mathbb Z$ if $x\in \mathcal O_{\mathbb K}$. In particular for a quadratic field $\mathbb K=\mathbb Q(\sqrt{d})$, where always $1\neq d\in\mathbb Z^{\times}$ is square-free, for $x=a+b\sqrt{d}$ we define the conjugate $x^{\ast}=a-b\sqrt{d}$. And thus the norm $N(x)$ equals $xx^{\ast}$.

\begin{rmk}
\label{Orders-rmk1}
Let $\mathbb K$ be an algebraic number field.

(1) If we have two orders $R_1\leq R_2\leq \mathcal O_{\mathbb K}$ and for some $x\in R_1$ we have $N(x)=p$, for some $p\in \Pr \mathbb Z$, then $x\in \Pr R_1$ and $x\in \Pr R_2$ as well. But in general the two sets of primes will not be equal. For example take $\mathbb K=\mathbb Q(\sqrt{5})$ and the non-maximal order $R=\mathbb Z[\sqrt{5}]$ in $\mathcal O_{\mathbb K}=\mathbb Z[\omega ]$, where $\omega =(1+\sqrt{5})/2$. Then the number $\mathfrak p=2+\omega$ is a prime in $\mathcal O_{\mathbb K}$, but at the same time it is not an element of $R$.

(2) For $\mathbb K$ and $R$ as in the above lemma in \cite[Thm. 0.1]{Sch} we have also proved this claim: If $\mathfrak q\in R$ satisfies $N(\mathfrak q)=a$, where $a\in \mathbb Z^{\times}$ is a composite number which is not of the form $a=p^k$ for some $p\in \Pr \mathbb Z$ and $k\geq 2$, then $\mathfrak q\not\in \Pr R$. Related to that note the following. Suppose $\mathfrak q\in R$ is such that $N(\mathfrak q)=p^k$, for some $p\in \Pr \mathbb Z$ and $k\geq 2$, so that there is no $x\in R$ and $\ell <k$ satisfying $N(x)=p^{\ell}$. Then we can conclude that $\mathfrak q$ is an irreducible element, but in general it will not be a prime. Namely as $R$ is a Noetherian ring, every element of $R$ can be written as a product of irreducible ones; see \cite[Thm. 3.2.2]{AW}. Hence it is clear that $\mathfrak q\in \Irr R$. On the other hand consider again $R=\mathbb Z[\sqrt{5}]$. It is easy to see that there is no $x\in R$ satisfying $N(x)=2c$, for some odd integer $c$. Thus it follows that for example the element $\gamma =3+\sqrt{5}$ is irreducible. But $\gamma$ is not a prime. Namely we have that $\gamma \gamma ^{\ast} =-4=2(-2)$, but it is easy to check that $\gamma \notdiv 2$. \end{rmk}

Having in mind what we have observed about the orders in algebraic number fields, let us state the following result which is a direct consequence of our Theorem \ref{mainTHM}.

\begin{thm}
\label{Orders-thm1}
Let $\mathbb K$ be an algebraic number field and $R$ be an order in $\mathbb K$. Let $f=c_nX^n+\cdots +c_1X+c_0$ be a polynomial in $R[X]$. Suppose there is some $\mathfrak p\in \Pr R$ such that: {\sf (C1)} $\ord _{\mathfrak p}(c_n)=0$; {\sf (C2)} $\gcd (n,\ord _{\mathfrak p}(c_0))=1$; {\sf (C3)} $(n-j)\ord _{\mathfrak p}(c_0) \leq n\ord _{\mathfrak p}(c_j)$, for every $1\leq j\leq n-1$. Then for $n\leq 6$ the polynomial $f$ is irreducible over $R$.
\end{thm}

\noindent
Note that when the considered order $R$ is not a UFD, this theorem (for $n\leq 6$) presents  nontrivial generalization of the Eisenstein-Dumas criterion given in the Introduction; see also Theorem \ref{Orders-thm2} given below. Let us give an illustrative example.

\begin{ex}
\label{Orders-ex1}
Consider the ring $R=\mathbb Z[\sqrt{10}]$. We claim that the polynomial
\begin{equation}
\label{Orders-ex1-label1}
f=X^n +(b\, 33!)X +45!
\end{equation}
is irreducible over $R$ for any $b\in\mathbb Z$ and $n\leq 6$. Here first observe that $R=\mathcal O_{\mathbb K}$ for $\mathbb K=\mathbb Q(\sqrt{10})$, where it is well known that $R$ is not a UFD. Further, by knowing some basic facts about generalized Pell equations one can show that $p=31$ is the minimal natural prime number such that the Diophantine equation $N(w)=p$ has solutions in $x,y\in \mathbb Z$, where $w=x+y\sqrt{10}$. For example, $w=\mathfrak p_1= 11+3\sqrt{10}$ and $w=\mathfrak p_2=29+9\sqrt{10}$ are two solutions. And therefore $\mathfrak p_1,\mathfrak p_2\in \Pr R$. Now take $\mathfrak p$ to be any of these two primes. We claim that $\ord _{\mathfrak p}(45!)=1$. To see this suppose that $\mathfrak p^2| 45!$; i.e., that $45! =\mathfrak p^2z$ for some $z\in R$. If we define $\ell =45!/31$ it is clear that $\mathfrak p \notdiv \ell$. And therefore by the equality $\mathfrak p\mathfrak p^{\ast}\ell =\mathfrak p^2z$ we conclude that necessarily $\mathfrak p|\mathfrak p^{\ast}$. But this is not the case and so the proposed equality holds; cf. the proof of the corollary below. At the same time obviously $\ord _{\mathfrak p}(b\, 33!)\geq 1$. By the previous theorem our claim about $f$ follows.
\end{ex}

In fact for the polynomial $f$ of the previous example we can show even more. But in order to be more precise we need the following corollary; here for the map $\ord _p$ see Lemma \ref{lemU1}. Note that for particular polynomials under consideration this corollary will be more appropriate for use than the last theorem.

\begin{cor}
\label{Orders-cor1}
Let $d\in \mathbb Z$ be square-free and $\mathbb K=\mathbb Q(\sqrt{d})$ be the corresponding quadratic field. Let $R$ be an order in $\mathbb K$ and $f=c_nX^n+\cdots +c_1X+c_0$ be a polynomial in $R[X]$. Assume that $2\neq p\in \Pr \mathbb Z$ is such that $\gcd (p,d)=1$ and there is some $\mathfrak q\in R$ satisfying $N(\mathfrak q)=p$. Also assume: {\sf (1)} $\ord _p(N(c_n))=0$; {\sf (2)} $\gcd (n,\ord _p(c_0))=1$; {\sf (3)} $(n-j)\ord _p(c_0) \leq n\ord _p(c_j)$, for every $1\leq j\leq n-1$. Then for $n\leq 6$ the polynomial $f$ is irreducible over $R$.
\end{cor}

\begin{proof}
As an auxiliary fact observe that there is no unit $u$ such that $\mathfrak qu=\mathfrak q^{\ast}$; i.e., $\mathfrak q$ does not divide $\mathfrak q^{\ast}$. First consider the easier case $d\not\equiv 1\pmod{4}$. Then suppose to the contrary and let us write $u=s+t\sqrt{d}$ and $\mathfrak q=\alpha +\beta \sqrt{d}$. Then we would have that $\alpha s+\beta dt=\alpha$ and $\beta s+\alpha t=-\beta$. Having in mind the Cramer's rule, define
$$
\Delta _0= \begin{vmatrix} \alpha & \beta d \\ \beta & \alpha \end{vmatrix}=N(\mathfrak q)=p \quad \text{ and } \quad \Delta _1= \begin{vmatrix} \alpha & \beta d \\ -\beta & \alpha \end{vmatrix} =2\alpha ^2 -p.
$$
As we have both $p\neq 2$ and $s=\Delta _1/\Delta _0\in\mathbb Z$, it is immediate that necessarily $p|\alpha$. Hence, using the fact that $\gcd (p,d)=1$, it follows that $p|\beta$ as well. But then $p^2|N(\mathfrak q)$; which cannot be. Now let $d\equiv 1\pmod{4}$ and then put $\omega =(1+\sqrt{d})/2$. Again suppose that for some $u=s+t\omega \in R^{\ast}$ and $\mathfrak q=\alpha +\beta \omega$ we have $\mathfrak qu=\mathfrak q^{\ast}$. It is immediate that then
$$
\alpha s+\beta t(d-1)/4= \alpha +\beta \quad \text{ and } \quad \beta s +(\alpha +\beta )t= -\beta .
$$
Here define
$$
\Delta _0= \begin{vmatrix} \alpha & \beta (d-1)/4 \\ \beta & \alpha +\beta \end{vmatrix} , \qquad \Delta _1= \begin{vmatrix} \alpha +\beta & \beta (d-1)/4 \\ -\beta & \alpha +\beta \end{vmatrix} , \qquad \Delta _2= \begin{vmatrix} \alpha & \alpha +\beta \\ \beta & -\beta \end{vmatrix}.
$$
We compute that
\begin{equation}
\label{Orders-cor1-label1}
\Delta _0=\alpha ^2 +\alpha \beta -\beta ^2\Bigl( \frac{d-1}{4}\Bigr) =N(\mathfrak q)=p,
\end{equation}
and also $\Delta _1=2\alpha ^2 +3\alpha \beta +\beta ^2-p$ and $\Delta _2=-\beta (\beta +2\alpha )$. As we have that $t=\Delta _2/\Delta _0\in\mathbb Z$ it follows that $p|\beta$ or $p|\beta +2\alpha$. In the former case by the above equality (\ref{Orders-cor1-label1}) it is clear that $p|\alpha$ as well. And so again $p^2|N(\mathfrak q)$, which is impossible. Next assume that $p$ does not divide $\beta$, but does $\beta +2\alpha$. And then rewrite the equality (\ref{Orders-cor1-label1}) as
$$
\alpha ^2 -\alpha \beta +\beta (\beta +2\alpha ) -\beta ^2\Bigl( \frac{d-1}{4}+1\Bigr) =p.
$$
Hence it is clear that $p$ divides
$$
4\alpha ^2 -4\alpha \beta -\beta ^2d-3\beta ^2= 2(\Delta _1+p) -5(2\alpha\beta +\beta ^2 ) -\beta ^2d,
$$
and therefore that $p|\beta$ as well; a contradiction.

Now define $\ell =\ord _p(c_0)$. As we have $\mathfrak q\mathfrak q^{\ast}=p$ it follows that both $\mathfrak q^{\ell}| c_0$ and $(\mathfrak q^{\ast})^{\ell} |c_0$. If $\ord _{\mathfrak q}(c_0)=\ell$, then using the fact that clearly $\ord _{\mathfrak q}(c_j)\geq \ord _p(c_j)$ for every $1\leq j\leq n-1$, the claim of the corollary follows by Theorem \ref{Orders-thm1} for the prime $\mathfrak p=\mathfrak q$. (See the proof of theorem which follows for more details.) The other possibility is that $\mathfrak q^{\ell +1} |c_0$. But then, using that $\mathfrak q\notdiv \mathfrak q^{\ast}$, it is easy to conclude that necessarily $\ord _{\mathfrak q^{\ast}}(c_0)=\ell$. Thus we can again apply Theorem \ref{Orders-thm1}, this time for the prime $\mathfrak p=\mathfrak q^{\ast}$.
\end{proof}

Let $f$ be given as in (\ref{Orders-ex1-label1}). If $d\in \mathbb N$ is square-free and we have some $p$ in $\Pi =\{ \pm 23, \pm 29, \pm 31\}$ and $\mathfrak p$ in $R=\mathbb Z[\sqrt{d}]$ such that $N(\mathfrak p)=p$, then for $n\leq 6$ the polynomial $f$ is irreducible over $R$. For example this is so for any $2\leq d\leq 10$. Namely we have the following: $N(5+\sqrt{2})=23$,
$N(2+3\sqrt{3})=-23$, $N(6+\sqrt{5})=31$, $N(1+2\sqrt{6})=-23$ and $N(6+\sqrt{7})=29$. Further let $\mathcal S$ be the set of all square-free positive integers $d<100$ satisfying $d\equiv 1\pmod{4}$. This set has $19$ elements, where the corresponding orders $R=\mathbb Z[\sqrt{d}]$ are not UFDs. For $12$ of these values $d$, when $d$ is not in the subset $\mathcal S_1= \{ 17,21, 29,37, 61,85, 89\}$ of $\mathcal S$, we have that $f$ is irreducible over $R$. Namely we have the following: $N(6+\sqrt{13})=23$, $N(8+\sqrt{33})=31$, $N(8+\sqrt{41})=23$, $N(36+5\sqrt{53})=29$, $N(22+3\sqrt{57})=-29$,   $N(6+\sqrt{65})=-29$, $N(10+\sqrt{69})=31$, $N(60+7\sqrt{73})=23$, $N(10+\sqrt{77})=23$, $N(8+\sqrt{93})=-29$ and $N(39+4\sqrt{97})=-31$. Also note that for example there are no $p\in\Pi$ and $x\in \mathbb Z[\sqrt{17}]$ so that $N(x)=p$; and similarly for other values $d\in\mathcal S_1$. Thus there is no help here of our corollary.

However the previous corollary can be further generalized. Here note that if our Conjecture is true, then the following theorem will take a form which presents a significant generalization of the classical Eisenstein-Dumas irreducibility criterion. Besides in many interesting situations we will have a polynomial $f\in\mathbb Z[X]$ and a $G_{\mathbb K}$-invariant order $R$ in a normal algebraic number field $\mathbb K$. And then it will be worth to know whether $f$ is irreducible over $R$. Having such setting this theorem might be of great help.

\begin{thm}
\label{Orders-thm2}
Let $\mathbb K$ be a normal algebraic number field of degree $m$ and $R$ be an order in $\mathbb K$. Suppose that $R$ is $G_{\mathbb K}$-invariant. Suppose $p\in \Pr \mathbb Z$ is such that there is some $\mathfrak p\in R$ whose norm equals $p$; and so $\mathfrak p\in\Pr R$. Let $f=c_nX^n+\cdots +c_1X+c_0$ be a polynomial in $R[X]$. Further suppose that $\gcd (n,m)=1$ and also for $\ell =\ord _p(c_0)$ the following three conditions: {\sf (1)} $\ord _p(N(c_n))= 0= \ord _p(N(c_0/p^{\ell}))$; {\sf (2)} $\gcd (n,\ell )=1$; {\sf (3)} $(n-j)\ell \leq n\ord _p(c_j)$, for every $1\leq j\leq n-1$. Then for $n\leq 6$ the polynomial $f$ is irreducible over $R$.
\end{thm}

\begin{proof}
Recall that two elements $a,b\in R$ are associated, which we denote by $a\sim b$, if $a=ub$ for some $u\in R^*$. For any $x\in R$ define $\mathcal A(x)=\{ \sigma (x) \mid \sigma \in G_{\mathbb K}\}$. As $R$ is $G_{\mathbb K}$-invariant we have that $\mathcal A(x)\subseteq R$. Clearly, $\sim$ is an equivalence relation on $\mathcal A(x)$. The elements of the corresponding quotient set $\mathcal A(x)/\!\!\sim$ will be denoted by $[y]$, for $y\in \mathcal A(x)$. It is easy to see that every two classes have the same cardinality; i.e., that
$$
\card \, [y] =\card \, [\sigma (y)], \quad \text{ for } \sigma \in G_{\mathbb K}.
$$
Thus in particular for $\mathfrak p\in R$ there is some $e\in\mathbb N$ so that $\card \, [\sigma (\mathfrak p)] =e$, for every $\sigma\in G_{\mathbb K}$. If we define $k$ to be the cardinality of $\mathcal A(\mathfrak p)/\!\!\sim$, then $ke=m$.

Now we have to show that the conditions {\sf (C1)}\! -- {\sf (C3)} of Theorem \ref{Orders-thm1} do hold. Concerning the first one suppose that $\ord _{\mathfrak p}(c_n)\geq 1$. Then there is some $x_n\in R$ so that $c_n=\mathfrak px_n$ and therefore $N(c_n)=pN(x_n)$, which contradicts to the first equality in {\sf (1)}.

Next we prove that {\sf (C2)} holds. Here $c_0=p^{\ell}x_0$, for some $x_0\in R$ such that $p\notdiv x_0$. We claim that $\mathfrak p\notdiv x_0$ as well. Namely suppose to the contrary; i.e., that $x_0=\mathfrak pw_0$ for some $w_0\in R$. But then we deduce that $p$ divides $N(x_0)$, which contradicts to the second equality in {\sf (1)}. Now choose $k$ automorphisms $\tau _i\in G_{\mathbb K}$ so that
$$
\{ [\tau _i(\mathfrak p)] \mid i=1,\ldots ,k\} =\mathcal A(\mathfrak p)/\!\!\sim ,
$$
where we can take $\tau _1$ to be the identity. Then there is some $v_0\in R^*$ so that
$$
c_0 =\Bigl( \prod _{\sigma \in G_{\mathbb K}} \sigma (\mathfrak p)^{\ell} \Bigr) x_0= \prod _{i=1}^k \bigl( \tau _i (\mathfrak p)^e\bigr) ^{\ell} x_0v_0= \mathfrak p^{e\ell} \Omega ,
$$
where we put $\Omega = \prod _{i=2}^k \bigl( \tau _i (\mathfrak p)^e\bigr) ^{\ell} x_0v_0$. Observe that by the $G_{\mathbb K}$-invariance of $R$ we have that $\Omega \in R$. As a conclusion we have that
\begin{equation}
\label{Orders-thm2-label1}
\ord _{\mathfrak p}(c_0)= e\ell .
\end{equation}
Now assume that there is some $\pi \in \Pr \mathbb Z$ such that $\pi$ divides $\gcd (n,\ord _{\mathfrak p}(c_0))$. We distinguish two possibilities. If $\pi |e$, then it follows that $\pi$ divides $\gcd (n,m)=1$; which is impossible. The second possibility is that $\pi |\ell$, which is again impossible by the condition {\sf (2)}. Thus we have proved that {\sf (C2)} holds.

In order to show {\sf (C3)} define $\ell _j=\ord _p(c_j)$, for $1\leq j\leq n-1$. So we have that $c_j=p^{\ell _j}y_j$, for some $y_j\in R$. Analogously as before we conclude that in particular $\mathfrak p^{e\ell _j}$ divides $c_j$, in $R$; i.e., that
\begin{equation}
\label{Orders-thm2-label2}
\ord _{\mathfrak p}(c_j)\geq  e\ell _j.
\end{equation}
Now the condition {\sf (3)} of the theorem reads as $(n-j)\ell \leq n\ell _j$, which is by the equality (\ref{Orders-thm2-label1}) equivalent to the inequality
$$
(n-j)\ord _{\mathfrak p}(c_0) \leq ne\ell _j.
$$
But then by (\ref{Orders-thm2-label2}) it is immediate that the condition {\sf (C3)} does hold.
\end{proof}

Every quadratic field $\mathbb K=\mathbb Q(\sqrt{d})$ is normal. Also it is well known that for $d\not\equiv 1\pmod{4}$ every order in $\mathbb K$ is of the form $R_k^{(1)}=\{ x+ky\sqrt{d} \mid x,y\in\mathbb Z\}$, while for $d\equiv 1\pmod{4}$ it is of the form $R_k^{(2)}=\{ x+ky\omega \mid x,y\in\mathbb Z\}$, where $\omega =(1+\sqrt{d})/2$. Thus it is clear that every order in $\mathbb K$ is $G_{\mathbb K}$-invariant. Further for any normal algebraic number field $\mathbb K$ its ring of integers $\mathcal O_{\mathbb K}$ is obviously $G_{\mathbb K}$-invariant. Therefore the following corollary is now obvious. Observe that in particular if the class number $h_{\mathbb K}$ of the field under consideration is greater than $1$, and so $\mathcal O_{\mathbb K}$ is not a UFD, then this corollary presents a novel useful result. Of course there are infinitely many such number fields. For examples of three normal cubic fields $\mathbb K$ having $h_{\mathbb K}\geq 2$ see \cite[App$.$ B, Table B.4]{C}; these have the discriminants $1957$, $2597$ and $2777$.

\begin{cor}
\label{Orders-cor2}
Let $\mathbb K$ be a normal algebraic number field of degree $m$. Suppose $p\in \Pr \mathbb Z$ is such that there is some $\mathfrak p\in \mathcal O_{\mathbb K}$ satisfying $N(\mathfrak p)=p$. Let $f=c_nX^n+\cdots +c_1X+c_0$ be a polynomial in $\mathcal O_{\mathbb K}[X]$. Further suppose that $\gcd (n,m)=1$ and for $\ell =\ord _p(c_0)$ the conditions {\sf (1)}\! -- {\sf (3)} of the previous theorem do hold. Then for $n\leq 6$ the polynomial $f$ is irreducible over $\mathcal O_{\mathbb K}$.
\end{cor}

Concerning Theorem \ref{Orders-thm2} and its corollary stated above observe the following. Suppose $\mathbb K$ is a normal algebraic number field of degree $\geq 2$ and $R$ is a $G_{\mathbb K}$-invariant order. It is possible to have a polynomial $f\in R[X]$ so that some of the conditions of 
Theorem \ref{Orders-thm2} do not hold; and therefore we cannot apply the theorem. But we can have a larger order $R\subset \widetilde{R}\subset \mathcal O_{\mathbb K}$ so that it is possible to conclude via our theorem that $f$ is irreducible over $\widetilde{R}$. Of course then $f$ is irreducible over $R$ as well. A simple but quite instructive example for that is the following one.

\begin{ex}
\label{Orders-ex1A}
Consider the quadratic field $\mathbb K=\mathbb Q(\sqrt{11})$ and, using the notation of the paragraph before the previous corollary, the order $R=R^{(1)}_{5^3\cdot 17}$ in $\mathbb K$. Also define the number 
$$
c_0= 999250 + 1846625\sqrt{11} = 5^3(2\cdot 7\cdot 571 + 11\cdot 17\cdot 79 \sqrt{11})
$$
of $R$ and then the cubic polynomial $f=X^3-c_0\in R[X]$. First note that a unique positive $p\in \Pr \mathbb Z$ satisfying $\ord _p(c_0)\in\mathbb N$, when $c_0$ is understood as an element of $\mathcal O_{\mathbb K}$, is $p=5$. Namely for such $p$ we would have some $w\in \mathcal O_{\mathbb K}$
so that $c_0=pw$. And then $N(c_0)= p^2N(w)$. It is easy to compute that $N(c_0)= -5^6\cdot 1327^3$, where $1327\in \Pr \mathbb Z$. Thus it is clear that necessarily $p\in \{ 5,1327\}$, where it is easy to rule out the possibility $p=1327$. 

Next observe that if we consider $c_0$ as an element of $\mathcal O_{\mathbb K}$, then $\ell = \ord _5(c_0)=3$ and so we cannot apply our theorem. More precisely, $f$ is reducible over $\mathcal O_{\mathbb K}$ as now $f$ can be written as $X^3-z_0^3$, for $z_0= 10+55\sqrt{11}$.

Further consider $c_0$ as an element of $R$. It is immediate that now we have that $\ord _5(c_0)=0$ and so again we cannot apply the theorem. A simple trick is to take $R\subset \widetilde{R}= R^{(1)}_{5^2\cdot 17}$. It is an easy exercise to check that here $\ord _5(c_0)=1$. And then we can conclude via our theorem thet $f$ is irreducible over $\widetilde{R}$; and so $f$ is irreducible over $R$ as well.
\end{ex}

Suppose $f$ is a polynomial defined over some order $R$ in a number field $\mathbb K$. In order to apply our results on reducibility of $f$ over $R$ one has to have a nontrivial insight into the set of primes $\Pr R$. But in general this might be a notoriously difficult problem. In particular this involves a problem to understand for which primes $p\in \Pr \mathbb Z$ we can find some (prime) $\mathfrak p\in \Pr R$ so that for the corresponding norm $N=N_{\mathbb K|\mathbb Q}$ we have $N(\mathfrak p)=p$. Finally, related to the last theorem it is reasonable to ask the following.
\begin{question}
{\em Given a normal algebraic number field $\mathbb K$, can we obtain some new (non-maximal) orders in $\mathbb K$ that are $G_{\mathbb K}$-invariant?}
\end{question}

Here we present a nontrivial observation about this question which might be interesting in its own right.

\begin{prop}
\label{Orders-prop1}
Let $\mathbb K$ be a normal algebraic number field of degree $m$. Assume that $\mathbb K=\mathbb Q(\alpha )$ for some $\alpha \in\mathbb K$ and let $\mu$ be the minimal polynomial of $\alpha$; where we can assume that $\mu \in \mathbb Z[X]$. Let $\alpha =\alpha _1,\alpha _2,\ldots ,\alpha _m$ be all the roots of $\mu$. Define the sets $\Gamma _k= \{ k\alpha _1,\ldots ,k\alpha _m\}$, for $k\in\mathbb N$, and then the rings $R_k= \mathbb Z[\Gamma _k]$; i.e., $R_k$ is the subring of $\mathbb C$ generated by $\Gamma _k$ over $\mathbb Z$. Then every $R_k$ is a $G_{\mathbb K}$-invariant order in $\mathbb K$.
\end{prop}

\begin{proof}
We will first show that
$$
\sigma (R_k)=R_k, \quad \text{ for every $\sigma \in G_{\mathbb K}$}.
$$
For that purpose observe that it is sufficient to show the inclusion $\sigma (R_k)\subseteq R_k$. Now define $\mathcal S$ to be the set of all subrings $S\leq \mathbb C$ which contain $\mathbb Z\cup \Gamma _k$. What we have to prove is the equality of sets
\begin{equation}
\label{Orders-prop1-label1}
\mathcal S= \{ \sigma (S) \mid S\in\mathcal S\} ,
\end{equation}
for every $\sigma$. Clearly $\sigma (S)$ is a subring of $\mathbb C$ containing $\mathbb Z$. And further we have that $\Gamma _k=\sigma (\Gamma _k)\subseteq \sigma (S)$. So it follows that $\sigma (S)\in\mathcal S$ for every $S\in\mathcal S$; i.e, we have the inclusion $\{ \sigma (S) \mid S\in\mathcal S\} \subseteq \mathcal S$. Hence it is immediate that (\ref{Orders-prop1-label1}) holds.

Now write
$$
\mu =X^m +a_{m-1}X^{m-1}+\cdots +a_1X+a_0, \qquad a_i\in\mathbb Z.
$$
Also define
$$
\mu _1 =\prod _{j=2}^m (X-\alpha _j) =X^{m-1}- \mathcal E _1X^{m-2}+\cdots +(-1)^{m-1} \mathcal E _{m-1},
$$
where $\{ \mathcal E _j\}$ are elementary symmetric functions in $\alpha _2,\ldots ,\alpha _m$. Thus we have that $\mu =(X-\alpha )\mu _1$ and hence that
$$
\alpha \mathcal E _{j-1} +\mathcal E _j = (-1)^j a_{m-j}, \quad \text{ for $j=1,\ldots ,m$},
$$
where we put $\mathcal E _0=1$ and $\mathcal E _m=0$. Obviously $\alpha \in R_1$. Further we have that $\alpha ^2 =\mathcal E _2 -\alpha a_{m-1} -a_{m-2}$ and therefore $\alpha ^2\in R_1$ as well. Similarly we see that every $\alpha ^j\in R_1$. Hence we conclude that $\mathbb Z[\alpha ]\subseteq R_1$. Next let $\preceq$ be the total ordering on $\mathbb N_0^m$, defined in the paragraph following Lemma \ref{ba1}; with there considered $\boldsymbol{\nu}\in \mathbb N_0^m$ taken to be $(1,\ldots ,1)$. For any $I=(i_1,\ldots ,i_m)$ define $\alpha _I=\alpha _1^{i_1}\cdots \alpha _m^{i_m}$, and then for fixed $z\in\mathbb Z$ the cyclic modules $\mathbb Zz\alpha _I$. Also define the set of multi-indices
$$
\mathcal I(m)=\{ I=(i_1,\ldots ,i_m) \mid i_j<m, \text{ for every $j$}\} ,
$$
and then the $\mathbb Z$-modules
$$
M_k=\sum _{I\in\mathcal I(m)}\mathbb Z k^{|I|} \alpha _I, \quad \text{ for $k\in \mathbb N$};
$$
where in particular for $\boldsymbol{0}=(0,\ldots ,0)$ we take $\alpha _{\boldsymbol{0}}=1$. Given any $I,J\in\mathcal I(m)$, define
$$
\Pi =(k^{|I|} \alpha _I)(k^{|J|} \alpha _J) = k^{|I+J|} \alpha _{I+J} = k^{|I+J|} \alpha _1^{i_1+j_1} \cdots \alpha _m^{i_m+j_m}.
$$
If for example $m\leq i_1+j_1<2m-1$, then
$$
\alpha _1^{i_1+j_1} = \alpha _1^{i_1+j_1-m}\Bigl( -\sum _{t=0}^{m-1} a_t\alpha _1^t\Bigr) = - \sum _{t=1}^m a_{m-t}\alpha _1^{i_1+j_1-t}.
$$
By performing further such computations we easily deduce that $\Pi \in M_k$. And therefore that every $M_k$ is a subring of $\mathbb C$.

Our next step is to show that moreover
\begin{equation}
\label{Orders-prop1-label2}
R_k=M_k, \quad \text{ for every $k\in\mathbb N$}.
\end{equation}
Indeed, as it is clear that the ring $M_k$ contains the set $\Gamma _k$, it follows that $R_K\subseteq M_k$. On the other hand for $I=(i_1,\ldots ,i_m)$ we have that
$$
k^{|I|} \alpha _I= (k\alpha _1)^{i_1} \cdots (k\alpha _m)^{i_m} \in R_k,
$$
and thus $M_k\subseteq R_k$ as well.

Our final step is to show thet every ring $R_k$ is an order in $\mathbb K$. For that purpose first note that by the inclusions $\mathbb Z[\alpha ] \subseteq R_1\subseteq \mathcal O_{\mathbb K}$ we know that $R_1$ is an order. For $k>1$ we have to show that the index $(R_1:R_k)$, of the additive group $R_k$ in the additive group $R_1$, is finite. More precisely, in our argument we will explore the fact (\ref{Orders-prop1-label2}). Now let $r\in \mathbb N$ be such that $r+1$ is the cardinality of the set $\mathcal I(m)$. Furthermore let us totally order this set of indices as
$$
\boldsymbol{0} =I_0 \preceq I_1\preceq \cdots \preceq I_r=\rho ,
$$
where we put $\rho =(m-1,\ldots ,m-1)$. And the define the $\mathbb Z$-modules $X_0=Y_0=\mathbb Z$, $X_1=\mathbb Z +\mathbb Z\alpha _m$,  $Y_1=\mathbb Z +\mathbb Zk\alpha _m$ and in general
$$
X_j=\sum _{i=0}^j \mathbb Z\alpha _{I_i} \quad \text{ and } \quad Y_j=\sum _{i=0}^j \mathbb Zk^{|I_i|}\alpha _{I_i}, \quad \text{ for $0\leq j\leq r$}.
$$
Observe that in particular $X_1/Y_1 =\{ t\alpha _m+Y_1 \mid 0\leq t<k\}$, and so $\card (X_1/Y_1)\leq k$. (Moreover the last inequality is in fact the equality.)  Further we note that
$$
X_2/(X_1+\mathbb Zk^{|I_2|}\alpha _{I_2}) \cong (X_2/Y_2) \bigl/(X_1+\mathbb Zk^{|I_2|}\alpha _{I_2})/Y_2.
$$
Denote the left-hand side of the above isomorphism by $A$ and observe that $\card (A)\leq k$ as well. Next define the module $W=\mathbb Z + \mathbb Zk^{|I_2|}\alpha _{I_2}$, and then also
$$
B=(X_1+\mathbb Zk^{|I_2|}\alpha _{I_2})/Y_2 =(W+\mathbb Z\alpha _{I_1})\bigl/ (W+\mathbb Zk^{|I_1|}\alpha _{I_1}).
$$
Obviously $B$ is a finite additive group too. Thus we have finite groups $A$ and $B$ such that $A\cong (X_2/Y_2)\bigl/ B$. And therefore $X_2/Y_2$ is finite as well. By inductive argument it follows that the groups $X_j/Y_j$ are finite, for every $j\in \{ 1,\ldots ,r-1\}$. Now note that
$$
M_1=X_{r-1} +\mathbb Z\alpha _{\rho} \quad \text{ and } \quad M_k=Y_{r-1} +\mathbb Zk^{|\rho |}\alpha _{\rho}.
$$
In particular let $0=x_0,x_1,\ldots ,x_s\in X_{r-1}$ be such that
$$
X_{r-1}/Y_{r-1} =\{ x_t+Y_{r-1} \mid 0\leq t\leq s\} .
$$
Any element $m\in M_1$ can be written as $m=x+b\alpha _{\rho}$, for some $x\in X_{r-1}$ and $b\in\mathbb Z$. At the same time there is some $v\in \{ 0,1,\ldots ,s\}$ so that $x-x_v\in Y_{r-1}$. And also there is a unique $j\in \{ 0,1,\ldots , k^{|\rho |}-1\}$ so that $b\equiv j\pmod{ k^{|\rho |}}$. Therefore we have that $(b-j)\alpha _{\rho}\in \mathbb Z  k^{|\rho |}\alpha _{\rho}$. And then
$$
x-x_v + (b-j) \alpha _{\rho} \in M_k,
$$
which further implies that
$$
m+M_k = x_v +j\alpha _{\rho} +M_k.
$$
As a conclusion we see that
$$
M_1/M_k =\{  x_v +j\alpha _{\rho} +M_k \mid v\in \{ 0,1,\ldots ,s\} \text{ and } 0\leq j<  k^{|\rho |} \} .
$$
Hence, $M_1/M_k$ is a finite additive group, as we had to show.
\end{proof}

\begin{rmk}
\label{Orders-rmk2}
Note that for $m=2$ and $d\not\equiv 1\pmod{4}$  our rings $R_k$ of the proposition are the before mentioned orders $R_k^{(1)}$ in $\mathbb K=\mathbb Q(\sqrt{d})$, while for $d\equiv 1\pmod{4}$ we have the corresponding orders $R_{2k}^{(2)}$. And obviously then for different values of $k$ we obtain the different orders. Related to that we believe that for the above considered orders $R_k$ we have $R_k\neq R_{\ell}$ if and only if $k\neq \ell$. But we do not know how to deal with this in a general setting.
\end{rmk}

We conclude the paper with one more result about the reducibility question for polynomials over orders in pure cubic fields. Note that these cubic fields will not be normal. Next we illustrate it via some examples.

\begin{prop}
\label{Orders-prop2}
Let $\delta \in\mathbb Z$ be cube-free and $\alpha$ be the real cube root of $\delta$. Consider the order $R=\mathbb Z[\alpha ]$ in the pure cubic field $\mathbb K=\mathbb Q(\alpha )$. Assume that $3\neq p\in \Pr \mathbb Z$ is such that $\gcd (p,\delta )=1$ and there is some $\mathfrak p\in \Pr R$ satisfying $N(\mathfrak p)=p$. Further suppose that if $\mathfrak p=a+b\alpha +c\alpha ^2$, for some integers $a,b$ and $c$, then we have the following:
\begin{equation}
\label{Orders-prop2-label1}
p\notdiv ab-c^2 \delta \quad \text{ or } \quad p\notdiv ac-b^2.
\end{equation}

Let $f=c_nX^n+\cdots c_1X+c_0$ be a polynomial in $R[X]$. Define $\ell =\ord _p(c_0)$ and assume that the same three conditions {\sf (1)}\!  -- {\sf (3)} of Theorem \textup{\ref{Orders-thm2}} do hold here as well. Then for $n\leq 6$ the polynomial $f$ is irreducible over $R$.
\end{prop}

First we prove one auxiliary result, which is in a sense similar to the observation, given in the proof of Corollary \ref{Orders-cor1}, that there considered prime $\mathfrak q$ does not divide its conjugate $\mathfrak q^*$.

\begin{lem}
\label{Orders-lem2}
Let $\delta , \alpha, R, \mathbb K, p$ and $\mathfrak p$ be as in the above proposition. Also assume that for the prime element $\mathfrak p$ the condition \textup{(\ref{Orders-prop2-label1})} holds.
Then $\mathfrak p^2$ does not divide $p$ in $R$.
\end{lem}

\begin{proof}
It is easy to see that the corresponding norm $N=N_{\mathbb K|\mathbb Q}$ equals
\begin{equation}
\label{Orders-lem2-label1}
N(x+y\alpha +z\alpha ^2)= x^3 +\delta y^3 +\delta ^2z^3 -3\delta xyz,
\end{equation}
for any $x,y,z\in\mathbb Q$. Now suppose to the contrary, that $\mathfrak p|p$, and let $w=x_0+y_0\alpha +z_0\alpha ^2\in R$ be such that $\mathfrak p^2w=p$. Then define the matrix
$$
M=\begin{pmatrix} a^2+2bc\delta & \delta (b^2+2ac) & \delta (c^2\delta +2ab) \\
   c^2\delta +2ab & a^2+2bc\delta & \delta (b^2+2ac) \\
   b^2+2ac & c^2\delta +2ab & a^2+2bc\delta \end{pmatrix}
$$
and the column vectors $\boldsymbol{v} =(x_0,y_0,z_0)^t$ and $\boldsymbol{b} =(p,0,0)^t$, where $t$ denotes the transpose. Then by the above equality $\mathfrak p^2w=p$ we obtain the system
$
M\boldsymbol{v}= \boldsymbol{b},
$
of three equations in the unknowns $x_0,y_0$ and $z_0$. Now define $M_i$ to be the matrix obtained so that the $i$-th column of $M$ is replaced by $\boldsymbol{b}$. And let then $\Delta _0=\det M$ and $\Delta _i=\det M_i$ for $1\leq i\leq 3$. Using the equality $p=a^3+\delta b^3 +\delta ^2c^3 -3\delta abc$, with a little effort one can check that $\Delta _0=p^2$. Also, if we define
 \begin{equation*}
  \begin{split}
  \Sigma _1 & = a^4 +3b^2c^2\delta ^2 -2ab^3\delta -2ac^3\delta ^2 , \\
  \Sigma _2 & = 2bc^3\delta ^2 -3a^2c^2\delta +2a^3b -b^4\delta , \\
  \Sigma _3 & = c^4\delta ^2 +3a^2b^2 - 2a^3c - 2b^3c\delta ,
   \end{split}
  \end{equation*}
then $\Delta _i=(-1)^{i-1} p\Sigma _i$, for $1\leq i\leq 3$. Using Cramer's rule one more time we in particular obtain that
$$
x_0=\Delta _1/\Delta _0= \Sigma _1/p =a+ \frac{3(ab-c^2\delta )(ac-b^2)\delta}{p} \in\mathbb Z.
$$
Taking into account that $y_0=\Delta _2/\Delta _0$ and $z_0=\Delta _3/\Delta _0$ are integers as well, we obtain at once that $p$ divides both $\Sigma _1$ and $\Sigma _2$.

Now assume that $p|ab-c^2\delta$. And then rewrite $\Sigma _1$ as
$$
\Sigma _1= 2(a^2-bc\delta )(ab-c^2\delta ) - \delta (ac-b^2)^2.
$$
Hence we get that also $p| ac-b^2$, which is impossible. The other possibility is that $p|ac-b^2$. This time we rewrite $\Sigma _2$ as
$$
\Sigma _2= -2(a^2-bc\delta )(ac-b^2) + (ab-c^2\delta )^2.
$$
Thus it follows that also $p|ab-c^2\delta$, which again cannot be.
\end{proof}

\begin{proof}[Proof of Proposition \textup{\ref{Orders-prop2}}]
Analogously as in the proof of Theorem \ref{Orders-thm2} we have to see that for $\mathfrak p$ the conditions {\sf (C1)}\! -- {\sf (C3)} of Theorem \ref{Orders-thm1} do hold. For {\sf (C1)} it is clear. In order to get the other two conditions it will suffice to show that
\begin{equation}
\label{Orders-prop2-label2}
\ord _{\mathfrak p}(c_0)= \ell .
\end{equation}
Clearly, $\ord _{\mathfrak p}(c_0)\geq \ell$. So assume that $\mathfrak p^{\ell +1}$ divides $c_0$ in $R$. And let $w\in R$ be such that $c_0=\mathfrak p^{\ell +1}w$. At the same time let $y_0\in R$ be such that $c_0=p^{\ell} y_0$, where by the condition {\sf (1)} we know that $p\notdiv N(y_0)$. Thus in particular it follows that $p^{\ell}y_0= \mathfrak p^{\ell +1}w$. As by the previous lemma we have that $\mathfrak p^2\notdiv p$, it is immediate that $\mathfrak p$ necessarily divides $y_0$ in $R$. But this means that $y_0=\mathfrak pr$, for some $r\in R$. Hence we deduce that $N(y_0) =pN(r)$, which cannot be. Thus we have (\ref{Orders-prop1-label2}) proved, and then it is clear that {\sf (C2)} holds. Besides the condition {\sf (C3)} holds as well; cf. the corresponding argument in the proof of Theorem \ref{Orders-thm2}.
\end{proof}

\begin{ex}
\label{Orders-ex2}
Suppose $k\in\mathbb Z$ is odd and square-free such that $k^2-4$ is not divisible by $9$; e.g., $|k|\in \{ 1,3,5,13,15,\ldots\}$. Define $\delta =4k$ and let $\alpha$ be the real cube root of $\delta$. By \cite[Thm$.$ 6.4.13]{C} we know that $(1,\alpha ,\alpha ^2/2)$ is an integral basis of $\mathbb K=\mathbb Q(\alpha )$. And thus $R=\mathbb Z[\alpha ]$ is a non-maximal order in $\mathbb K$; in particular it is not a UFD. Suppose there is a prime $p\in \Pr \mathbb Z$ for which we can find some (prime) $\mathfrak p=a+b\alpha +c\alpha ^2$ in $R$ so that (\ref{Orders-prop2-label1}) holds and $N(\mathfrak p)=p$; where $N$ is as in (\ref{Orders-lem2-label1}). Now let $f=c_nX^n+\cdots +c_1X+c_0$ be a polynomial in $\mathbb Z[X]$, for $n\leq 6$. Put $\ell =\ord _p(c_0)$ and assume the following: {\sf (1)} $\ord _p(c_n)=0$; {\sf (2)} $\gcd (n,\ell )=1$; {\sf (3)} $(n-j)\ell \leq n\ord _p(c_j)$ for every $1\leq j\leq n-1$. Then $f$ is irreducible over $R$. Note that for $n\leq 6$ this is a nontrivial generalization of the classical Eisenstein-Dumas criterion, stated in the Introduction.

For example suppose that for the polynomial $f$ as above there exists a prime $p\in\{ 11,13,19,29\}$ so that: $(\boldsymbol{\dagger})$  $p|c_j$ for every $0\leq j<n$, while both $p\notdiv c_n$ and $p^2\notdiv c_0$. We claim that $f$ is irreducible as a polynomial over the non-UFD over-ring $R=\mathbb Z[\sqrt[3]{12}]$ of $\mathbb Z$. Namely here $\delta =12$ and $\alpha =\sqrt[3]{12}$. If we define $\mathfrak p_1=\alpha -1$, $\mathfrak p_2=\alpha +1$, $\mathfrak p_3=7-3\alpha$ and  $\mathfrak p_4=5-2\alpha$, then $N(\mathfrak p_1)=11$, $N(\mathfrak p_2)=13$, $N(\mathfrak p_3)=19$ and $N(\mathfrak p_4)=29$. Clearly for every $\mathfrak p_i$ the condition (\ref{Orders-prop2-label1}) is fulfilled. Thus our claim follows by Proposition \ref{Orders-prop2}. Also suppose that for example there is a prime $p\in \{ 5,11,17,31\}$ for which the above condition $(\boldsymbol{\dagger})$ holds. Then it is easy to conclude that $f$ is irreducible over the non-UFD ring $R=\mathbb Z[\sqrt[3]{4}]$. Similar conclusions can be obtained for some other orders $R$ in pure cubic fields $\mathbb Q(\sqrt[3]{\delta})$.
\end{ex}

\end{document}